\documentclass{aims} 
\usepackage{amsmath}
\usepackage{paralist}
\usepackage[misc]{ifsym}
\usepackage{epsfig} 
\usepackage{epstopdf} 
\usepackage[colorlinks=true]{hyperref}
\hypersetup{urlcolor=blue, citecolor=red}
\allowdisplaybreaks

\def\currentvolume{X}
 \def\currentissue{X}
  \def\currentyear{200X}
   \def\currentmonth{XX}
    \def\ppages{X-XX}
     \def\DOI{10.3934/ammc.2024021}

\usepackage{amssymb, amsfonts, amsthm}
\usepackage{mathtools}
\usepackage{bbm}
\usepackage{mathabx} 
\usepackage{tabularx}
\usepackage{ifthen}
\usepackage{multicol}
\usepackage[usestackEOL]{stackengine}
\usepackage{makecell}
\usepackage[inline]{enumitem}

\usepackage{tikz}
\usepackage{pgfplots}
\usetikzlibrary{positioning, shapes, fit, calc, external, arrows, matrix}
\pgfplotsset{compat=1.15}

\mathtoolsset{showonlyrefs} 

\newcommand{\matop}[1]{\mathbf{#1}}  
\newcommand{\mathvec}[1]{\boldsymbol{{#1}}} 
\newcommand{\eye}[1]{\mathbf{I}_{\scriptscriptstyle{#1}}}  
\def\gramcc{\matop{G}_{\scriptscriptstyle\Delta_j\Delta_j}}
\def\gramcf{\matop{G}_{\scriptscriptstyle\Delta_j\nabla_j}}
\def\predictor{\matop{P}}

\def\mgard{\mathsf{MGARD}}

\DeclareMathOperator{\diag}{diag}
\DeclareMathOperator{\spann}{span}

\makeatletter
\newcommand*{\mathreflect}[1]{\binrel@{#1}\binrel@@{\mathpalette\math@reflect{#1}}}
\newcommand*{\math@reflect}[2]{\reflectbox{\m@th$#1#2$}}
\newcommand*{\mathrotate}[3][]{\binrel@{#3}\binrel@@{\vphantom{#3}\mathpalette\math@rotate{{#1}{#2}{#3}}}}
\newcommand*{\math@rotate}[2]{\math@@rotate#1#2}
\newcommand*{\math@@rotate}[4]{
  \sbox\z@{$\m@th#1#4$}%
  \smash{\makebox[\wd\z@]{\rotatebox[#2]{#3}{$\m@th#1#4$}}}%
}
\makeatother
\newcommand*{\obslash}{\mathreflect{\oslash}}
\newcommand*{\overt}{\mathrotate[origin=c]{90}{\ominus}}

\makeatletter
\newcommand{\smallbullet}{} 
\DeclareRobustCommand\smallbullet{%
  \mathord{\mathpalette\smallbullet@{0.7}}%
}
\newcommand{\smallbullet@}[2]{%
  \vcenter{\hbox{\scalebox{#2}{$\m@th#1\bullet$}}}%
}
\makeatother

\newtheorem{theorem}{Theorem}[section]

\newtheorem{lemma}[theorem]{Lemma}

\theoremstyle{definition}
\newtheorem{definition}[theorem]{Definition}
\newtheorem{remark}[theorem]{Remark}

\title[Lifting MGARD: construction of (pre)wavelets on the interval]
{Lifting MGARD: construction of (pre)wavelets on the interval using polynomial predictors of arbitrary order} 

\author[V. Reshniak, E. Ferguson, Q. Gong, N. Vidal, R. Archibald and S. Klasky]{}


\thanks{This manuscript has been authored by UT-Battelle, LLC, under contract DE-AC05-00OR22725 with the US Department of Energy (DOE). The US government retains and the publisher, by accepting the work for publication, acknowledges that the US government retains a non-exclusive, paid-up, irrevocable, world-wide license to publish or reproduce the submitted manuscript version of this work, or allow others to do so, for US government purposes. DOE will provide public access to these results of federally sponsored research in accordance with the DOE Public Access Plan (http://energy.gov/downloads/doe-public-access-plan).}

\thanks{$^*$Corresponding author: Viktor Reshniak}

\begin{document}
\maketitle

\centerline{\scshape
Viktor Reshniak$^{{\href{mailto:reshniakv@ornl.gov}{\textrm{\Letter}}}*1}$,
Evan Ferguson$^{{\href{mailto:efergu2@uw.edu}{\textrm{\Letter}}}2}$,
Qian Gong$^{{\href{mailto:gongq@ornl.gov}{\textrm{\Letter}}}1}$,
}
\centerline{\scshape
Nicolas Vidal$^{{\href{mailto:vidaln@ornl.gov}{\textrm{\Letter}}}1}$,
Richard Archibald$^{{\href{mailto:archibaldrk@ornl.gov}{\textrm{\Letter}}}1}$
and Scott Klasky$^{{\href{mailto:klasky@ornl.gov}{\textrm{\Letter}}}1}$
}

\medskip

{\footnotesize
\centerline{$^1$Oak Ridge National Laboratory, Oak Ridge, TN, USA}
} 

\medskip

{\footnotesize
 \centerline{$^2$University of Washington, Seattle, Washington, USA}
}

\bigskip



\begin{abstract}
    MGARD (MultiGrid Adaptive Reduction of Data) is an algorithm for compressing and refactoring scientific data, based on the theory of multigrid methods. 
    The core algorithm is built around stable multilevel decompositions of conforming piecewise linear $C^0$ finite element spaces, enabling accurate error control in various norms and derived quantities of interest.
    In this work, we extend this construction to arbitrary order Lagrange finite elements $\mathbb{Q}_p$, $p \geq 0$, and propose a reformulation of the algorithm as a lifting scheme with polynomial predictors of arbitrary order. 
    Additionally, a new formulation using a compactly supported wavelet basis is discussed, and an explicit construction of the proposed wavelet transform for uniform dyadic grids is described.
\end{abstract}


\section{Introduction}
\label{sec:intro}

Compression of scientific data is known to be particularly challenging as it is widely regarded as effectively incompressible due to the intrinsically high entropy of its floating-point value representation \cite{lakshminarasimhan2013isabela,SAYOOD20061}.
Lossless compressors achieve at most mediocre performance on such datasets \cite{Zhao2020}.
At the same time, scientific data from simulations and experiments is often inherently lossy and can tolerate some error-controlled loss in its accuracy.
Wavelets are particularly appealing for this task because of their excellent localization properties in both spatial and frequency domains \cite{STEPHANE20091}.

Wavelet compression is closely related to the time-scale multiresolution analysis which is rooted in Laplacian pyramids \cite{BURT1987671}, hierarchical filters \cite{Roos1988,Viergever1993}, multirate and polyphase filter banks \cite{fliege1994multirate,Minh2012,Vaidyanathan1990}.
Similarly to the wavelet transform, these methods decompose signals into components at different scales, capturing fine details and low-frequency structures separately.
This approach facilitates efficient encoding by exploiting the sparsity of transformed signals.
Conventional wavelets are constructed with translation invariant filters defined on regular nested dyadic grids, see filter bank diagram in Figure~\ref{fig:filter_bank_lifting}.
Invertibility, smoothness, vanishing moments, and the support size of the wavelet basis functions are the main design criteria commonly imposed in the frequency domain for such regularly sampled signals.

\begin{figure}[tb]
    \centering
    \def\w{0.45\textwidth}
    \def\h{0.37\textwidth}
    \stackinset{c}{0pt}{t}{-12pt}{}{\includegraphics[height=\h]{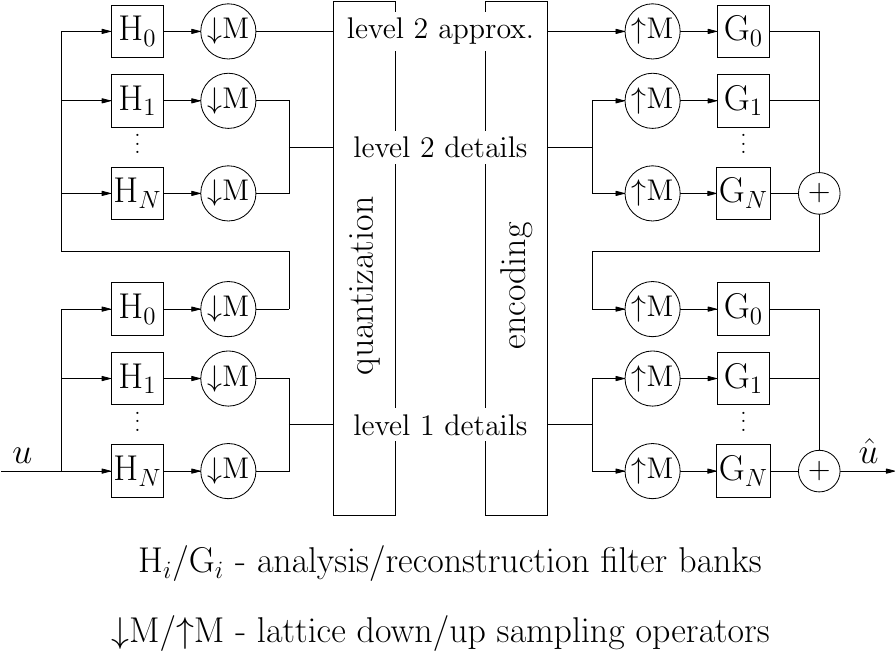}}
    \stackinset{c}{0pt}{t}{-12pt}{}{\includegraphics[height=\h]{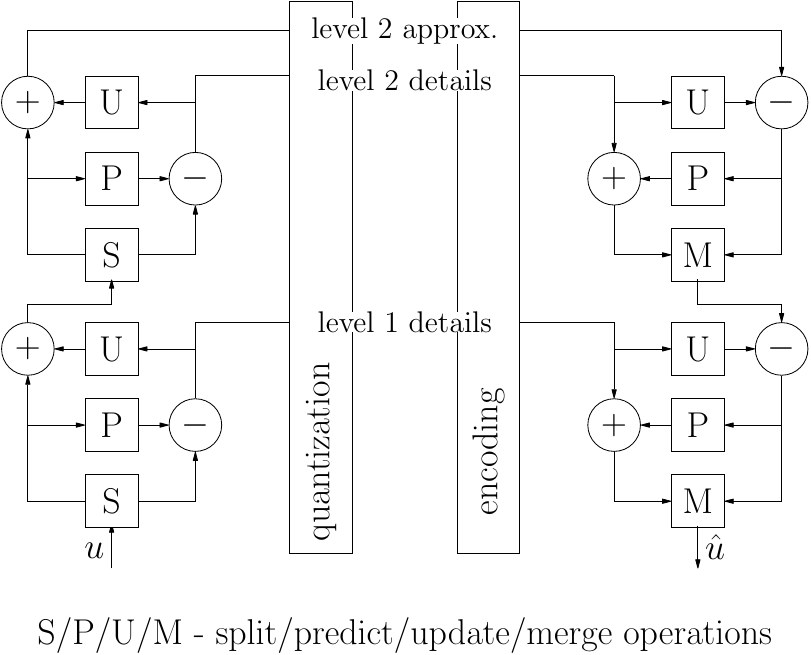}}
    \caption{Schematic illustration of the two-level decomposition/recomposition of a signal using equivalent subband coding filter bank (left) and lifting scheme (right) implementations.}
    \label{fig:filter_bank_lifting}
\end{figure}

Lifting provides an alternative approach for constructing wavelets directly in the signal domain using appropriate linear combinations of scaling functions \cite{Sweldens2000}. 
The scheme comprises the split/merge, predict, and update operations and is invertible by design.
After the signal is split into two components, the values in one are used to predict the values of another, and the prediction error is recorded as the detail coefficients, see Figure~\ref{fig:filter_bank_lifting}.
The quality of predictor is largely responsible for the energy compaction of this decomposition while the update operation is commonly designed to match the vanishing moments of the primal/reconstruction and dual/analysis wavelets, the property generally leading to better stability of the wavelet transform.
Factorization of classical wavelet decompositions into lifting steps also provides more efficient implementation over the filter bank construction~\cite{daubechies1998factoring}.
As a result, lifting is used as a coding algorithm in wavelet-based compressors including JPEG2000~\cite{Taubman2002} and SPECK~\cite{Pearlman2004} for size-bounded compression, Waverange~\cite{Kolomenskiy2022-po} and SPERR~\cite{Sperr2023} for error-controlled compression of structured scientific data, and HexaShrink~\cite{Peyrot2019} for compression of structured hexahedral volume meshes.

On a more abstract note, the construction of wavelets is intricately tied to the concept of stable decomposition of function spaces.
Considering a family of projectors $\mathcal{P}_j: V \to V_j$ with $V_{j}\subset V_{j+1} \subset V$ and $\mathcal{P}_{j}\mathcal{P}_{j+1}=\mathcal{P}_{j}$, an element $u\in V$ possesses a multiscale representation
\begin{align}\label{eq:multiscale_repr}
    u 
    = \sum_{j=0}^\infty (\mathcal{P}_j-\mathcal{P}_{j-1})u
    = \mathcal{P}_0u + \sum_{j=0}^\infty \mathcal{Q}_j u,
    \qquad
    \mathcal{Q}_j = \mathcal{P}_{j+1} - \mathcal{P}_{j}: V \to W_j.
\end{align}
The operator $\mathcal{Q}_j$ is also a projector since $\mathcal{Q}_j^2 = \mathcal{P}_{j+1}^2 - \mathcal{P}_{j+1}\mathcal{P}_{j} - \mathcal{P}_{j}\mathcal{P}_{j+1} + \mathcal{P}_{j}^2 = \mathcal{P}_{j+1} - \mathcal{P}_{j}$.
Consequently, since $\mathcal{P}_j(\mathcal{P}_{j+1}-\mathcal{P}_{j})=0$, the element $\mathcal{Q}_ju\in W_j$ can be viewed as the detail complement of $\mathcal{P}_{j}u\in V_{j}$ in $V_{j+1}$ leading to the direct sum decomposition of the space $V$ 
\begin{align*}
    V = V_0\bigoplus_{l=0}^{\infty}W_l.
\end{align*}
The choice of appropriate projectors is crucial to obtain decompositions with certain desirable properties such as the polynomial exactness of $V_j$ or the stability of the representation \eqref{eq:multiscale_repr} written in terms of the bases of the approximation and detail subspaces $V_j:=\spann \{\varphi_{j,k}:k\in\Delta_j\}$, $W_j:=\spann \{\psi_{j,k}:k\in\nabla_j\}$ for some appropriate index sets $\Delta_j/\nabla_j$
\begin{align}\label{eq:stability}
    u &= \sum_{k\in\Delta_0} \alpha_{0,k} \varphi_{0,k} + \sum_{j\geq0}\sum_{k\in\nabla_j} \beta_{j,k}\psi_{j,k}
    \qquad
    \text{s.t.}\quad \|u\|_V \sim \vvvert\boldsymbol{\alpha}\cup\boldsymbol{\beta}\vvvert,
\end{align}
where $\vvvert\cdot\vvvert$ is a corresponding discrete norm and $\|\cdot\|\sim\vvvert\cdot\vvvert$ denotes the usual norm equivalence, i.e., the existence of finite constants $\gamma,\Gamma$ with $\gamma\vvvert\boldsymbol{\alpha}\cup\boldsymbol{\beta}\vvvert \leq \|u\|_V \leq \Gamma\vvvert\boldsymbol{\alpha}\cup\boldsymbol{\beta}\vvvert.$

In the case of $V$ being a Hilbert space, a suitable candidate for $\mathcal{P}_j$ is an oblique projector defined as
\begin{align}\label{eq:projector}
    \mathcal{P}_ju = \sum_{k\in\Delta_j}\langle u,\tilde{\varphi}_{j,k}\rangle \varphi_{j,k},
\end{align}
where $\tilde{\varphi}$ is biorthogonal to $\varphi$, i.e., $\langle\varphi_{j,k},\tilde{\varphi}_{l,k'}\rangle=\delta_{k,k'}$.
It has been shown that biorthogonality is a necessary condition for the uniform stability of multiscale representations in the sense of the norm equivalence in~\eqref{eq:stability} when $V$ is $L^2(\Omega)$ \cite{cohen2003numerical,Dahmen1995}.
Here and below, $\Omega$ refers to some measure space of interest such as $\mathbb{R}^d$ or its compact subset.
When omitted, $\Omega$ should be understood from the context.
Note that biorthogonality condition eliminates the selection of interpolation operators 
\begin{align}\label{eq:interpolator}
    \mathcal{I}_ju=\sum_{k\in\Delta_j}u_{k}\varphi_{j,k},
    \qquad
    [\varphi_{j,k}]_{k'} = \delta_{k,k'},
\end{align}
as uniformly $L^2$-stable projectors, given that $\tilde{\varphi}_{j,k}$ is a Dirac distribution defined by $\langle u,\tilde{\varphi}_{j,k}\rangle:=u_{j,k}$.
In terms of the lifting scheme, this corresponds to the interpolating predictor without an update operation, underscoring once again the significance of the update operation in achieving stable multiresolution encoding.

Similar results also hold in $L^p$-Besov spaces $V:=B^{t}_{p,q}(\Omega)$ as given by
\begin{theorem}[\protect{\cite[Theorem~3.7.7]{cohen2003numerical}}]\label{th:Besov_norm}
    Assume that $\varphi,\tilde{\varphi}$ are compactly supported biorthogonal scaling functions with $\varphi\in L^r(\Omega)$, $\tilde{\varphi}\in L^{r'}(\Omega)$ for $r\in[1,\infty]$, $1/r+1/r'=1$, or that $\varphi\in C^0(\Omega)$ and $\tilde{\varphi}$ is a Radon measure in which case $r=\infty$.
    Then, for $0<p\leq r$, one has the norm equivalence
    \begin{align*}
        \|u\|_{B^{s}_{p,q}}\sim \|\mathcal{P}_0u\|_{L^p} + \left\|\left\{2^{sj}\|\mathcal{Q}_ju\|_{L^p}:j\geq0\right\}\right\|_{l^q},
    \end{align*}
    for all $s>0$ such that 
    \begin{align}\label{eq:s_range}
        d\left(\frac{1}{p}-\frac{1}{r}\right)<s<\min(t,n+1),
    \end{align}
    where $n$ is the order of polynomial reproduction in $V_j$ and $t$ is such that $\varphi\in B^{t}_{p,q_0}(\Omega)$ for some $q_0$.
\end{theorem}

The condition in \eqref{eq:s_range} defines the valid range of the smoothness of the space in which the norm is assessed.
This range is characterized by the data dimensionality~$d$, smoothness $t$ and polynomial reproduction $n$ of the basis $\varphi$, and $L^r$ stability of the projector in \eqref{eq:projector}.
Since $B^s_{2,2}=H^s$, Theorem~\ref{th:Besov_norm} is also immediately valid for the usual Sobolev spaces $H^s(\Omega)$ which, together with the $L^2$ stability result, gives
\begin{align}\label{eq:Sobolev_norm}
    \|u\|_{H^{s}}^2 \sim \sum_{k\in\Delta_0}|\alpha_{k}|^2 + \sum_{j\geq0}\sum_{k\in\nabla_j}2^{2sj}|\beta_{j,k}|^2,
    \quad d\left(\frac{1}{2}-\frac{1}{r}\right)\leq s<\min(t,n+1).
\end{align}
Such norm equivalences in $H^s$ have been extensively studied in the design of additive multilevel preconditioners of elliptic operators \cite{BornemannNorm1993,Oswald1994,XuSubspace1992}.

For piecewise polynomial nodal basis, it can be shown that $\varphi\in H^t$ with $t<n+1/2$, where $n$ is the order of the basis.
This simplifies the upper bound in \eqref{eq:Sobolev_norm}.
To simplify the lower bound, we need $r=2$ to get rid of the dimension $d$ yielding
\begin{align*}
    \varphi\text{ is }n\text{-th order Lagrange polynomial }
    \quad\to\quad 
    0\leq s<n+1/2.
\end{align*}
It is also worth noting that Theorem~\ref{th:Besov_norm} allows using interpolation projectors in \eqref{eq:interpolator} that are $L^p$ unstable for all $p<\infty$.
However, the condition in \eqref{eq:Sobolev_norm} becomes more restrictive due to its dependence on the dimension~$d$.
For instance, the piecewise linear interpolation corresponds to $n=1$, $r=\infty$, and $\varphi\in H^{3/2}$ resulting in $s~\in~\left(\frac{d}{2},\frac{3}{2}\right)$ if $d\leq2$ and no valid values for $d>2$.
These lower bounds once again highlight the special role of $L^2$-stable projectors.

MGARD (MultiGrid Adaptive Reduction of Data) is an error controlled compression algorithm that builds upon the theory of stable decomposition of function spaces discussed above \cite{AinsworthMGARD1D2018,AinsworthMGARDnD2019,GongMGARDSoftware2023}.
The main ingredient of MGARD is given by $L^2$-\textit{orthogonal} projectors that we denote $\mathcal{P}_j^o$ to make a distinction from the more general oblique projectors~$\mathcal{P}_j$ in \eqref{eq:projector}.
The apparent drawback of this approach is the global support of the dual basis $\tilde{\varphi}$.
However, it is partially eliminated by efficient linear-complexity solvers for inverting banded Gramm matrices of compactly supported primal basis $\varphi$ in one-dimensional and tensor-product spaces.
Using the equivalence relation in~\eqref{eq:Sobolev_norm}, MGARD is able to control the compression error in the corresponding Sobolev norms, as well as in $L^\infty$ norm and arbitrary linear quantities of interest (QoI)~\cite{AinsworthMGARDQoI2019}.
Further extensions include machine learning assisted error control of nonlinear QoI~\cite{Lee2022}, adaptive compression with feature preservation~\cite{GongClimate2023}, and compression of data defined on certain unstructured grids~\cite{AinsworthMGARDFE2020}.

The core MGARD algorithm is designed assuming piecewise linear data representation.
As a result, the norm equivalence in \eqref{eq:Sobolev_norm} is valid only for $s\in[0,3/2)$ and the method cannot take advantage of higher regularity in data.
In this work, we extend this construction to arbitrary order Lagrange finite elements $Q_p$, $p\geq 0$, and propose a reformulation of the algorithm as a lifting scheme with variable order polynomial predictors.
Our primary theoretical contributions are as follows:
\begin{enumerate}
    \item we present a practical, constructive scheme for the wavelet decomposition of arbitrary-order piecewise polynomial spaces on the interval, 
    \item we introduce a new family of stable projection operators, which lead to compactly supported wavelet bases.
\end{enumerate}
The key practical contributions are as follows:
\begin{enumerate}[resume]
    \item the flexible selection of the wavelet order allows for better alignment with the intrinsic data regularity, leading to improved compression ratios,  
    \item the compact support of the wavelet bases facilitates a more efficient implementation.
\end{enumerate}
We anticipate that the proposed construction will serve as a foundation for further algorithmic extensions and more precise error control.

The paper is organized as follows.
In Section~\ref{sec:MGARD_transform}, we formulate MGARD transform as a lifting scheme, provide its matrix formulation, and describe refinement equations.
Section~\ref{sec:uniform_grids} is devoted to the specific construction for the uniform dyadic grids on the interval and tensor product grids.
Section~\ref{sec:results} contains numerical results demonstrating properties of the transform.
Section~\ref{sec:conclusion} concludes the discussion and provides several directions for future work.

\section{MGARD wavelet transform}
\label{sec:MGARD_transform}

Here we provide a rather general construction of the MGARD transform as a wavelet lifting scheme.
First, denote by $V(\Omega)$ the space of continuous functions defined on a compact set $\Omega\in\mathbb{R}^d$ and consider a sequence of nested grids
\begin{align*}
    \mathcal{G}_0\subset\hdots\subset\mathcal{G}_{j-1}\subset\mathcal{G}_j\subset\hdots\subset\mathcal{G}_L\subset\Omega,
    \qquad
    \mathcal{G}_j:=\{x_{k}:k\in\Delta_j\},
\end{align*}
for some nested nodal index sets $\Delta_{j-1}\subset\Delta_j$.
For each grid $\mathcal{G}_j$, consider two subdivisions into
\begin{itemize}
    \item a collection $\mathcal{T}_j$ of non-overlapping elements,
    \item a collection $\mathcal{S}_j$ of non-overlapping subdomains such that each $\sigma\in\mathcal{S}_j$ is a union of one or more connected elements in $\mathcal{T}_j$.
    A natural choice for $\mathcal{S}_j$ is $\mathcal{T}_k,\;k\leq j$. 
\end{itemize}
Figure~\ref{fig:subdivision} illustrates one possible subdivision of the grid in two domains.

\begin{figure}[tb]
    \centering
    \scalebox{0.8}{
    \tikzsetnextfilename{grid}
    \tikzset{
        dot/.style   = {circle, fill, draw=black, minimum size=3pt, inner sep=0pt, outer sep=0pt},
        fdot/.style  = {circle, fill, draw=gray,  minimum size=2pt, inner sep=0pt, outer sep=0pt},
    }
    \begin{tikzpicture}
        \node [dot] (c)    at (0,0) {};
        
        \node [dot] (p0)   at (2,0) {};
        \node [dot] (p60)  at (1,1.72) {};
        \node [dot] (p120) at (-1,1.72) {};
        \node [dot] (p180) at (-2,0) {};
        \node [dot] (p240) at (-1,-1.72) {};
        \node [dot] (p300) at (1,-1.72) {};

        \node [fdot] (fp0)   at (1,0) {};
        \node [fdot] (fp60)  at (0.5,0.86) {};
        \node [fdot] (fp120) at (-0.5,0.86) {};
        \node [fdot] (fp180) at (-1,0) {};
        \node [fdot] (fp240) at (-0.5,-0.86) {};
        \node [fdot] (fp300) at (0.5,-0.86) {};
        \node [fdot] (fp30)  at (1.5,0.86) {};
        \node [fdot] (fp90)  at (0,1.72) {};
        \node [fdot] (fp150) at (-1.5,0.86) {};
        \node [fdot] (fp210) at (-1.5,-0.86) {};
        \node [fdot] (fp270) at (0,-1.72) {};
        \node [fdot] (fp330) at (1.5,-0.86) {};

        \draw[black, very thick]  (c) -- (p0);
        \draw[black, very thick]  (c) -- (p60);
        \draw[black, very thick]  (c) -- (p120);
        \draw[black, very thick]  (c) -- (p180);
        \draw[black, very thick]  (c) -- (p240);
        \draw[black, very thick]  (c) -- (p300);
        
        \draw[black, very thick] (p0) -- (p60) -- (p120) -- (p180) -- (p240) -- (p300) -- (p0);

        \draw[gray]  (fp0)   -- (fp30)  -- (fp60)  -- (fp0);
        \draw[gray]  (fp60)  -- (fp90)  -- (fp120) -- (fp60);
        \draw[gray]  (fp120) -- (fp150) -- (fp180) -- (fp120);
        \draw[gray]  (fp180) -- (fp210) -- (fp240) -- (fp180);
        \draw[gray]  (fp240) -- (fp270) -- (fp300) -- (fp240);
        \draw[gray]  (fp300) -- (fp330) -- (fp0)   -- (fp300);
    \end{tikzpicture}    
    \hspace{2em}
    \tikzsetnextfilename{rect_grid}
    \begin{tikzpicture}
        \foreach \yy in {0,1.2,...,3.6} \foreach \xx in {0,2.0,...,8} \node [dot]  at (\xx,\yy) {};
        \foreach \yy in {0,0.3,...,3.6} \foreach \xx in {0,0.5,...,8} \node [fdot] at (\xx,\yy) {};

        \foreach \xx in {0,0.5,...,8}   \draw[gray] (\xx,0) -- (\xx,3.6);
        \foreach \yy in {0,0.3,...,3.6} \draw[gray] (0,\yy) -- (8,\yy);
        
        \foreach \xx in {0,2.0,...,8.0} \draw[black, very thick] (\xx,0) -- (\xx,3.6);
        \foreach \yy in {0,1.2,...,3.6} \draw[black, very thick] (0,\yy) -- (8,\yy);
    \end{tikzpicture} 
    }
    \caption{A subdivision of two domains into elements $\mathcal{T}$ (gray) and subdomains $\mathcal{S}$ (black).}
    \label{fig:subdivision}
\end{figure}
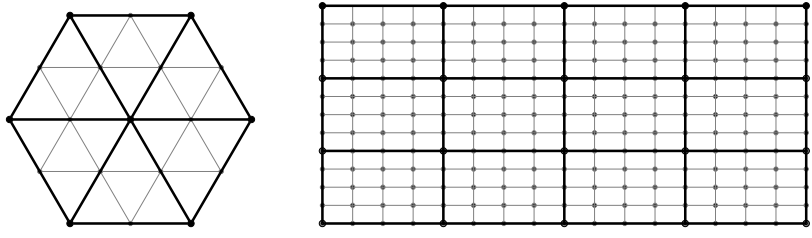

Denote by $\Delta^{\tau}_j:=\Delta_j\cap\tau$, i.e., those indices of $\Delta_j$ that are also in $\tau\in\mathcal{T}_j$.
Define $\Delta^{\sigma}_j:=\Delta_j\cap\sigma$ analogously.
To each $\sigma\in\mathcal{S}_j$, associate a discrete function space $V_j^{\sigma}$ with a nodal basis $\{\varphi^{\sigma}_{j,k}:k\in\Delta^{\sigma}_j\}$.
Given $V_j^{\sigma}$, define two function spaces on the whole grid $\mathcal{G}_j$: 
\begin{itemize}
    \item a `broken' space $V_{\mathcal{S}_j}:=\bigoplus_{\sigma\in\mathcal{S}_j} V^{\sigma}_j$ of functions that are continuous in each $\sigma\in\mathcal{S}_j$, but discontinuous across the subdomains,
    \item a space $V_j$ of continuous functions spanned by the global nodal basis $\{\varphi_{j,k}:k\in\Delta_j\}$, i.e., the basis of $V_{\mathcal{S}_j}$ except the basis vectors at the interface nodes are `glued' together.
\end{itemize}

\begin{remark}
    For the proposed construction, one needs to ensure the nestedness of the spaces $V_j\subset V_{j+1}$ and $V_{\mathcal{S}_j}\subset V_{\mathcal{S}_{j+1}}$.
    In one dimension, the natural choice is the space of piecewise polynomials on dyadically refined grids. 
    In this scenario, each polynomial on the coarse element is precisely represented by two polynomials of the same order on the elements of the finer grid. 
    This construction is also applicable to tensor product spaces. 
    A similar approach applies to certain simplicial meshes obtained through iterative refinement of the initial mesh via subdivision of simplex edges \cite{AinsworthMGARDFE2020}, see Figure~\ref{fig:subdivision}. 
    Consequently, the data values at the nodes of the coarsest simplex uniquely determine the linear function that can be exactly represented by simplexes at finer levels. 
    However, this method does not straightforwardly extend to general unstructured meshes. 
    One potential solution involves commencing with the final mesh and employing mesh coarsening to establish the mesh hierarchy. 
    The function spaces on the resulting polygonal meshes can then be defined by employing generalized barycentric coordinates, augmented with additional constraints to ensure hierarchical construction \cite{FLOATER2005623, HACKEMACK2018188}.
\end{remark}

Given a sequence of nested spaces $V_j$, consider three projection operators:
\begin{enumerate}
    \item the interpolation operator $\mathcal{I}_j:V\to V_j$
    \begin{align}\label{eq:interpolant}
        \mathcal{I}_ju = \sum_{k\in\Delta_j}u_{k}\varphi_{j,k},
    \end{align}
    \item the family of (oblique) projectors $\mathcal{P}_j:=\mathcal{A}_{\mathcal{S}_j}\mathcal{P}_{\mathcal{S}_j}:V\to V_j$ with the projectors $\mathcal{P}_{\mathcal{S}_j}:V\to V_{\mathcal{S}_j}$ defined in each subdomain as
    \begin{align}\label{eq:dg_projector}
        (\mathcal{P}_{\mathcal{S}_j} u,v)_{\sigma} &= (u,v)_{\sigma}, \quad\forall u\in V, v\in V^{\sigma}_j,
    \end{align}
    and the operator $\mathcal{A}_{\mathcal{S}_j}:V_{\mathcal{S}_j} \to V_j$ used to enforce the nodal continuity at the interface nodes
    \begin{align}\label{eq:smoother}
        \mathcal{A}_{\mathcal{S}_j}u &= \sum_{k\in\Delta_j}\{\!\{u_{k}\}\!\}\varphi_{j,k},
    \end{align}
    where $\{\!\{u_{k}\}\!\}$ is the average value, 
    \item the detail projection operator $\mathcal{Q}_j:=\mathcal{P}_{j+1}-\mathcal{P}_j:V\to W_j$ inducing the orthogonal decomposition of each $V_j$ into its approximation and detail subspaces $V_j=V_{j-1}\oplus W_{j-1}$.
\end{enumerate}
Note that when $S_j$ consists of the single subdomain, $\mathcal{P}_j$ is the $L^2$-orthogonal projector onto $V_j$.
We will denote such a projector as $\mathcal{P}_j^{cg}$ to emphasize its relation to the continuous Galerkin finite element approximation.
Similarly, we will use $\mathcal{P}_j^{dg}$ to denote the projector in \eqref{eq:dg_projector}-\eqref{eq:smoother} when $\mathcal{S}_j=\mathcal{T}_j$.
Figure~\ref{fig:projections} shows different projections into the space of piecewise quadratic functions defined on the grid with eight elements.

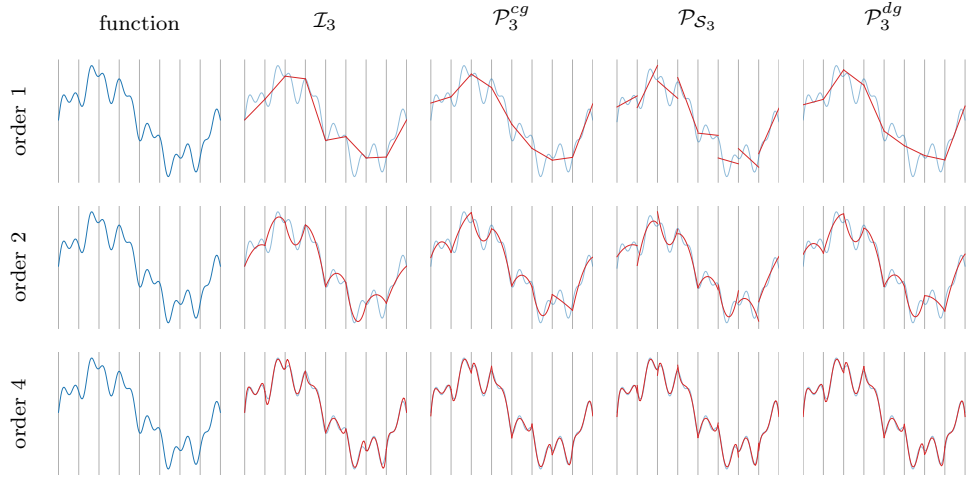
\begin{figure}[t]
    \centering
    \pgfplotsset{every tick label/.append style={font=\scriptsize}}
    \pgfplotsset{every axis label/.append style={font=\footnotesize}}
    \pgfplotsset{every axis title/.append style={font=\footnotesize}} 	
    
    \foreach \order in {1,2,4}
    {
        \begin{tikzpicture}[trim axis left,trim axis right]
            \begin{axis} [
                width=0.3\textwidth,
                ymode = normal,
                xticklabels={,,},
                yticklabels={,,},
                enlargelimits=false,
                axis on top,
                axis line style={draw=none},
                title=\ifthenelse{\equal{\order}{1}}{function}{},
                ylabel={order $\order$},
                tick style={draw=none},
                very thick ]
                \addplot graphics [points={(0,0) (1,1)}, includegraphics={trim=0 0 0 0,clip}] {{{images/fun}.pdf}};
            \end{axis}
        \end{tikzpicture}
        \foreach \proj/\tit in {interp/$\mathcal{I}_3$,CG/$\mathcal{P}_3^{cg}$,DG/$\mathcal{P}_{\mathcal{S}_3}$,DGCG/$\mathcal{P}_3^{dg}$}
        {
            \begin{tikzpicture}[trim axis left,trim axis right]
                \begin{axis} [
                    width=0.3\textwidth,
                    ymode = normal,
                    xticklabels={,,},
                    yticklabels={,,},
                    enlargelimits=false,
                    axis on top,
                    axis line style={draw=none},
                    title=\ifthenelse{\equal{\order}{1}}{\tit}{},
                    tick style={draw=none},
                    very thick ]
                    \addplot graphics [points={(0,0) (1,1)}, includegraphics={trim=0 0 0 0,clip}] {{{images/\proj_proj_\order}.pdf}};
                \end{axis}
            \end{tikzpicture}
        }
        \\[-1.5em]
    }
    \vspace{1.5em}
    \caption{Projections into the space of piecewise polynomials on the grid with $2^3$ elements.}
    \label{fig:projections}
\end{figure}

The crux of the MGARD wavelet transform is based on the following factorizations of $\mathcal{P}_j$ and~$\mathcal{Q}_j$
\begin{align}
    \label{eq:P_factorization}
    \mathcal{P}_{j} 
    &= \mathcal{I}_{j} \mathcal{P}_{j+1} + (\mathcal{P}_{j}-\mathcal{I}_{j}\mathcal{P}_{j+1})
    = \Big( \mathcal{I}_{j} + \mathcal{P}_{j}(I-\mathcal{I}_{j}) \Big) \mathcal{P}_{j+1},
    \\[1em]\label{eq:Q_factorization}
    \mathcal{Q}_{j}
    &:= \mathcal{P}_{j+1}-\mathcal{P}_{j}
    =\Big( I - \mathcal{I}_{j} - \mathcal{P}_{j}(I-\mathcal{I}_{j}) \Big) \mathcal{P}_{j+1}
    = \big(I-\mathcal{P}_{j}\big)\big(I-\mathcal{I}_{j}\big)\mathcal{P}_{j+1},
\end{align}
where we used $\mathcal{P}_{j}\mathcal{I}_{j}=\mathcal{I}_{j}$ which trivially holds for the projectors in \eqref{eq:interpolant} and \eqref{eq:dg_projector}-\eqref{eq:smoother}.

For any $v\in V_{j+1}$, we have $(v-\mathcal{I}_{j}v)\in V_{j+1}$ with
$$
v-\mathcal{I}_{j}v = 
\sum_{k\in\nabla_{j}}\big[v-\mathcal{I}_{j}v\big]_{k}\varphi_{j+1,k},
$$
and hence
\begin{align*}
    \big(I-\mathcal{P}_{j}\big)\big(I-\mathcal{I}_{j}\big)\sum_{k\in\Delta_{j+1}}v_{k}\varphi_{j+1,k}
    = \sum_{k\in\nabla_{j}}\big[v-\mathcal{I}_{j}v\big]_{k}\big(I-\mathcal{P}_{j}\big)\varphi_{j+1,k},
\end{align*}
where $\nabla_j:=\Delta_{j+1}\setminus\Delta_j$ is the surplus index set of those indices in $\Delta_{j+1}$ that are not in $\Delta_j$. 
It is convenient to write the above representation of $\mathcal{Q}_jv\in W_j$ for any $v\in V_{j+1}$ as an expansion
\begin{align}\label{eq:Q_application}
    \mathcal{Q}_{j}v := 
    \big(I-\mathcal{P}_{j}\big)\big(I-\mathcal{I}_{j}\big)v
    = \sum_{k\in\nabla_{j}}\beta_{j,k}\psi_{j,k}
\end{align}
with
\begin{align}\label{eq:wavelet_basis}
    \beta_{j,k} &:=\big[v-\mathcal{I}_{j}v\big]_{k}, 
    \qquad
    \psi_{j,k}:=\big(I-\mathcal{P}_{j}\big)\varphi_{j+1,k},
    \qquad
    k\in\nabla_j.
\end{align}
Hence, $\psi_{j,k}$ is the basis of the detail subpace $W_j$, and the wavelet coefficients $\beta_{j,k}$ quantify the mismatch between $v\in V_{j+1}$ and its interpolant $\mathcal{I}_{j}v\in V_j$ at the surplus nodes of $\nabla_j$.
By this definition, the coefficients $\beta_{j,k}$ are identically zero for all $v\in V_j$. 
Analogously, $\mathcal{P}_j\psi_{j,k}=0$ by construction reflecting the orthogonality of spaces $V_j\perp W_j$.

Similarly to \eqref{eq:Q_application}, one gets
\begin{align}\label{eq:P_application}
    \mathcal{P}_{j}v &:= 
    \Big( \mathcal{I}_{j} + \mathcal{P}_{j}(I-\mathcal{I}_{j}) \Big) \sum_{k\in\Delta_{j+1}}v_{k}\varphi_{j+1,k}
    = \sum_{k\in\Delta_{j}}v_k\varphi_{j,k} + \sum_{n\in\nabla_{j}}\beta_{j,n}\mathcal{P}_{j}\varphi_{j+1,n}
    \\\nonumber
    &=  \sum_{k\in\Delta_{j}}v_k\varphi_{j,k} + \sum_{n\in\nabla_{j}}\beta_{j,n}\sum_{k\in\Delta_j}u_{j,k,n}\varphi_{j,k}
    = \sum_{k\in\Delta_{j}} \left( v_k + \sum_{n\in\nabla_{j}}u_{j,k,n}\beta_{j,n} \right) \varphi_{j,k},
\end{align}
where $\{u_{j,k,n}:k\in\nabla_j\}$ are the coefficients in the nodal basis expansion of 
\begin{align}\label{eq:P_in_basis}
    \mathcal{P}_{j}\varphi_{j+1,n}
    = \sum_{k\in\Delta_j}u_{j,k,n}\varphi_{j,k},
    \quad n\in\nabla_j.
\end{align}

We now have all the components to formulate the MGARD transform as a lifting scheme using the \textit{Split}, \textit{Predict} and \textit{Update} operations detailed below.

\textbf{Split operation.}
Denote by $\boldsymbol{\alpha}_{j+1}:=\{[\mathcal{P}_{j+1}u]_{k}:k\in\Delta_{j+1}\}$ the vector of the nodal values of $\mathcal{P}_{j+1}v\in V_{j+1}$ at $\Delta_{j+1}$.
Given $\boldsymbol{\alpha}_{j+1}$, the \textit{Split} operation produces two vectors with the values of $\mathcal{P}_{j+1}v$ at the nodes of $\Delta_j\subset\Delta_{j+1}$ and $\nabla_j\subset\Delta_{j+1}$, namely $\boldsymbol{\alpha}_{j+1}^{\Delta}:=\{[\mathcal{P}_{j+1}v]_{k}:k\in\Delta_{j}\}$ and $\boldsymbol{\alpha}_{j+1}^\nabla:=\{[\mathcal{P}_{j+1}v]_{k}:k\in\nabla_{j}\}$.

\textbf{Predict operation.}
Using the nodal values of $\mathcal{P}_{j+1}v$ at $\Delta_j$, the \textit{Predict} operation estimates the values of $\mathcal{P}_{j+1}v$ at the nodes of $\nabla_j$ using the interpolant in \eqref{eq:interpolant}.
The detail coefficients are then calculated as the mismatch of this prediction
\begin{align*}
    \boldsymbol{\beta}_j = \boldsymbol{\alpha}_{j+1}^\nabla - \predictor_j\boldsymbol{\alpha}_{j+1}^\Delta,
\end{align*}
where $\boldsymbol{\beta}_{j}:=\{\beta_{j,k}:k\in\nabla_{j}\}$, and $\predictor_j\in\mathbb{R}^{|\nabla_j^\sigma|\times|\Delta_j^\sigma|}$ is the matrix representation of the interpolating predictor.

\textbf{Update operation.}
Given the nodal values $\boldsymbol{\alpha}_{j+1}^\Delta$ of $\mathcal{P}_{j+1}v$ at $\Delta_j$ and the detail coefficients $\boldsymbol{\beta}_{j}$ at $\nabla_j$, the \textit{Update} operation calculates the nodal values $\boldsymbol{\alpha}_j$ of $\mathcal{P}_{j}v$ at $\Delta_j$ using \eqref{eq:P_application}.
Specifically, the coefficients $\matop{U}^{\sigma}_{j}:=\{u^{\sigma}_{j,k,n}:k\in\Delta^{\sigma}_j,n\in\nabla^{\sigma}_j\}$ 
for the projections of the basis functions $\{\varphi^{\sigma}_{j+1,n}:n\in\nabla^{\sigma}_j\}$ in each subdomain $\sigma\in\mathcal{S}_j$ are given by
\begin{align}\label{eq:update_mat}
    \gramcc^{\sigma} \matop{U}^{\sigma}_{j} = \gramcf^{\sigma}
    \quad\to\quad
    \matop{U}^{\sigma}_{j} = \left(\gramcc^{\sigma}\right)^{-1}\gramcf^{\sigma},
\end{align}
where $\gramcc^{\sigma}$ and $\gramcf^{\sigma}$ are the Gram matrices with the entries
\begin{alignat*}{2}
    &[\gramcc^{\sigma}]_{n',m'} := \langle\varphi^{\sigma}_{j,n},\varphi^{\sigma}_{j,m}\rangle,
    \quad &&n\in\Delta^{\sigma}_j, m\in\Delta^{\sigma}_j,
    \\
    &[\gramcf^{\sigma}]_{n',m'} := \langle\varphi^{\sigma}_{j,n},\varphi^{\sigma}_{j+1,m}\rangle,
    \quad &&n\in\Delta^{\sigma}_j, m\in\nabla^{\sigma}_j,
\end{alignat*}
and $n',m'$ are the local (to the subdomain $\sigma$) indices of the global indices $n,m$.
The columns of the update matrix $\matop{U}^{\sigma}_{j}\in\mathbb{R}^{|\Delta_j^\sigma|\times|\nabla_j^\sigma|}$ contain the nodal values of the projected basis functions in each subdomain.
The global update matrix $\matop{U}_j\in\mathbb{R}^{|\Delta_j|\times|\nabla_j|}$ is then obtained by applying the averaging operator in \eqref{eq:smoother} to the columns of $\{\matop{U}^{\sigma}_{j}, \sigma\in \mathcal{S}_j\}$.

Consequently, the nodal values $\boldsymbol{\alpha}_j$ of $\mathcal{P}_{j}u$ at $\Delta_j$ are calculated as
\begin{align*}
    \boldsymbol{\alpha}_j = \boldsymbol{\alpha}_{j+1}^\Delta + \matop{U}_j \boldsymbol{\beta}_j.
\end{align*}

The steps in the described transform are fully invertible, as indicated in Figure~\ref{fig:MGARD_lifting}, with the obvious definition of the \textit{Merge} operation.
To summarize, the steps for the forward and inverse transforms are given in Table~\ref{tab:MGARD_steps}.

\begin{figure}[tb]
    \centering
    \def\w{\textwidth}
    \stackinset{c}{0pt}{t}{-12pt}{}{\includegraphics[width=\w]{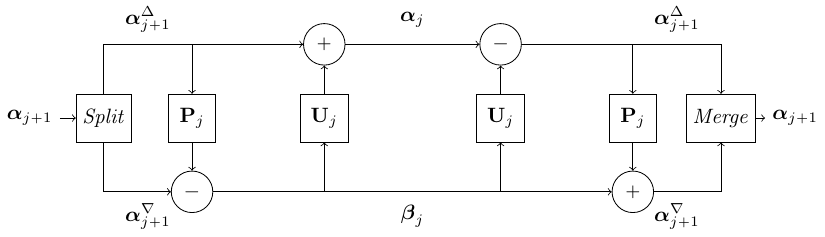}}
    \caption{Schematic illustration of MGARD lifting steps.}
    \label{fig:MGARD_lifting}
\end{figure}

\begin{table}[h]
    \centering
    \vspace{1em}
    \begin{tabularx}{\linewidth}{>{\def\fullwidthdisplay{\displayindent0pt \displaywidth\hsize}}Xp{0.05\linewidth}|p{0.05\linewidth}>{\def\fullwidthdisplay{\displayindent0pt \displaywidth\hsize}}X}
        \makecell{\textbf{Forward transform}, \\$\mgard_{j+1\to j}: \mathvec{\alpha}_{j+1}\to(\mathvec{\alpha}_j,\mathvec{\beta}_j)$} &&& \makecell{\textbf{Inverse transform}, \\$\mgard_{j\to j+1}^{-1}: (\mathvec{\alpha}_j,\mathvec{\beta}_j)\to\mathvec{\alpha}_{j+1}$}\\[2em]
        \hline
        {\begin{align}
            \mathvec{\alpha}_{j+1}^\Delta, \mathvec{\alpha}_{j+1}^\nabla &= \text{\textit{Split}}(\mathvec{\alpha}_{j+1})
            \\[1em]
            \mathvec{\beta}_j &= \mathvec{\alpha}_{j+1}^\nabla - \predictor_j\mathvec{\alpha}_{j+1}^\Delta
            \\[1em]
            \mathvec{\alpha}_j &= \mathvec{\alpha}_{j+1}^\Delta + \matop{U}_j \mathvec{\beta}_j
        \end{align}}
        & & &
        {\begin{align}
            \mathvec{\alpha}_{j+1}^\Delta &= \mathvec{\alpha}_j - \matop{U}_j \mathvec{\beta}_j
            \\[1em]
            \boldsymbol{\alpha}_{j+1}^\nabla &= \mathvec{\beta}_j + \predictor_j\mathvec{\alpha}_{j+1}^\Delta
            \\[1em]
            \mathvec{\alpha}_{j+1} &= \text{\textit{Merge}}(\mathvec{\alpha}_{j+1}^\Delta,\mathvec{\alpha}_{j+1}^\nabla)
        \end{align}}
    \end{tabularx}
    \caption{MGARD transform steps.}
    \label{tab:MGARD_steps}
\end{table}

\begin{remark}
It is important to highlight that the factorizations in \eqref{eq:P_factorization}-\eqref{eq:Q_factorization} remain valid for any pair of projectors satisfying $\mathcal{P}_{j}\mathcal{I}_{j}=\mathcal{I}_{j}$.
For example, if we substitute $\mathcal{P}_j$ with $\mathcal{I}_j$, we recover the hierarchical basis construction of Yserentant in \cite{BornemannNorm1993}. 
However, such replacement lacks stability, as indicated in Section~\ref{sec:intro}. 
On a positive note, this substitution leads to a straightforward implementation since the \textit{Update} operation is not required.
The advantage of introducing the interpolation operator $\mathcal{I}_j$ into the factorizations of $\mathcal{P}_j$ and $\mathcal{Q}_j$ becomes apparent from the expansions in \eqref{eq:Q_application}, \eqref{eq:P_application} which require the same number of degrees of freedom in total as the original projector $\mathcal{P}_{j+1}$.
Additionally, the nodal basis used in \eqref{eq:interpolant} enables a straightforward interpretation of data as the coefficients of such basis expansions.
\end{remark}

\subsection{Matrix formulation}

The following definitions are helpful to formulate the matrix representation of the transform.
\begin{definition}\label{def:binaryop}
    Define two binary operations:
    \begin{enumerate}
        \item $\blackdiamond:\mathbb{R}^{\smallbullet\times\Delta_j}\times\mathbb{R}^{\smallbullet\times\nabla_j}\to\mathbb{R}^{\smallbullet\times\Delta_{j+1}}$ creates a new matrix by merging the \textbf{columns} of the inputs according to the index sets $\Delta_j,\nabla_j\subset\Delta_{j+1}$,
        \item $\diamond:\mathbb{R}^{\Delta_j\times\smallbullet}\times\mathbb{R}^{\nabla_j\times\smallbullet}\to\mathbb{R}^{\Delta_{j+1}\times\smallbullet}$ creates a new matrix by merging the \textbf{rows} of the inputs according to the index sets $\Delta_j,\nabla_j\subset\Delta_{j+1}$.
    \end{enumerate}
    The following properties are trivial
    \begin{multicols}{2}
    \begin{enumerate}
        \item $\matop{c}(\matop{a}\blackdiamond\matop{b}) = \matop{c}\matop{a}\blackdiamond\matop{c}\matop{b}$,
        \item $(\matop{a}\diamond\matop{b})\matop{c} = \matop{a}\matop{c}\diamond\matop{b}\matop{c}$,
        \item $(\matop{a}\blackdiamond\matop{b})(\matop{c}\diamond\matop{d}) = \matop{a}\matop{c} + \matop{b}\matop{d}$,
        \item $[\matop{a},\matop{b}]^T\blackdiamond[\matop{c},\matop{d}]^T = [\matop{a}\blackdiamond\matop{c},\matop{b}\blackdiamond\matop{d}]^T$,
        \item $[\matop{a},\matop{b}]\diamond[\matop{c},\matop{d}] = [\matop{a}\diamond\matop{c},\matop{b}\diamond\matop{d}]$,
        \item $(\matop{a}\blackdiamond\matop{b})^T=\matop{a}^T\diamond\matop{b}^T$,
        \item $\matop{a}\diamond(\matop{b}+\matop{c}) = \matop{a}\diamond\matop{b} + \matop{a}\diamond\matop{c}$,
        \item $(\matop{a}+\matop{b})\blackdiamond\matop{c} = \matop{a}\blackdiamond\matop{c}+\matop{b}\blackdiamond\matop{c}$.
    \end{enumerate}
    \end{multicols}
\end{definition}

\begin{definition}\label{def:stackelemets}
    For any two matrices $\matop{a}\in\mathbb{R}^{m\times n}$, $\matop{b}\in\mathbb{R}^{p\times q}$, defi\textbf{}ne the following concatenation operators
    \begin{alignat*}{2}
        \matop{a}\overt\matop{b}
        &:=
        \begingroup 
        \setlength\arraycolsep{2pt}
        \begin{bmatrix}
            a_{00}   & \hdots & a_{0n}  &        &        \\
            \vdots   & \ddots & \vdots  &        &        \\
            a_{m0}   & \hdots & a_{mn}  &        &        \\
                     &        & b_{00}  & \hdots & b_{0q} \\
                     &        & \vdots  & \ddots & \vdots \\
                     &        & b_{p0}  & \hdots & b_{pq} \\
        \end{bmatrix},
        \endgroup
        \qquad&&
        \matop{a}\ominus\matop{b}
        :=
        \begingroup 
        \setlength\arraycolsep{2pt}
        \begin{bmatrix}
            a_{00}   & \hdots & a_{0n}  &         &        & \\
            \vdots   & \ddots & \vdots  &         &        & \\
            a_{m0}   & \hdots & a_{mn}  & b_{00}  & \hdots & b_{0q} \\
                     &        &         & \vdots  & \ddots & \vdots \\
                     &        &         & b_{p0}  & \hdots & b_{pq} \\
        \end{bmatrix},
        \endgroup
        \\[1em]
        \matop{a}\oplus\matop{b}
        &:=
        \begingroup 
        \setlength\arraycolsep{2pt}
        \begin{bmatrix}
            a_{00}   & \hdots & a_{0n}  &        &        \\
            \vdots   & \ddots & \vdots  &        &        \\
            a_{m0}   & \hdots & a_{mn}+b_{00}  & \hdots & b_{0q} \\
                     &        & \vdots  & \ddots & \vdots \\
                     &        & b_{p0}  & \hdots & b_{pq} \\
        \end{bmatrix},
        \endgroup
        \qquad&&
        \matop{a}\obslash\matop{b}
        :=
        \begingroup 
        \setlength\arraycolsep{2pt}
        \begin{bmatrix}
            a_{00}   & \hdots & a_{0n}  &         &        &        \\
            \vdots   & \ddots & \vdots  &         &        &        \\
            a_{m0}   & \hdots & a_{mn}  &         &        &        \\
                     &        &         & b_{00}  & \hdots & b_{0q} \\
                     &        &         & \vdots  & \ddots & \vdots \\
                     &        &         & b_{p0}  & \hdots & b_{pq} \\
        \end{bmatrix},
        \endgroup
    \end{alignat*}
    so that $\matop{a}\overt\matop{b}\in\mathbb{R}^{m+p\times n+q-1}$, $\matop{a}\ominus\matop{b}\in\mathbb{R}^{m+p-1\times n+q}$, $\matop{a}\oplus\matop{b}\in\mathbb{R}^{m+p-1\times n+q-1}$, and $\matop{a}\obslash\matop{b}\in\mathbb{R}^{m+p\times n+q}$.
\end{definition}

Using Definition~\ref{def:binaryop}, the matrix representations of the transform are given by
\begin{align}\label{eq:levelmatmgard}
    \begin{bmatrix}
        \mathvec{\alpha}_j\\[0.5em]
        \mathvec{\beta}_j
    \end{bmatrix}
    &=
    \mathsf{M}_j
    \cdot
    \mathvec{\alpha}_{j+1},
    \qquad
    \mathvec{\alpha}_{j+1} 
    =
    \mathsf{M}_j^{-1}
    \cdot
    \begin{bmatrix}
        \mathvec{\alpha}_j\\[0.5em]
        \mathvec{\beta}_j
    \end{bmatrix},
\end{align}
where 
\begin{align*}
    \mathsf{M}_j &:=
    \begin{bmatrix}
        \mathsf{A}_j\\[0.5em]
        \mathsf{B}_j
    \end{bmatrix},
    \qquad
    \mathsf{M}_j^{-1} := \begin{bmatrix} \mathsf{C}_j & \mathsf{D}_j \end{bmatrix}
\end{align*}
with
\begin{alignat*}{2}
    \mathsf{A}_j &:= \big( \eye{\Delta_j} - \matop{U}_j\predictor_j\big) \blackdiamond \matop{U}_j,
    \quad&&
    \mathsf{B}_j := (-\predictor_j)\blackdiamond\eye{\nabla_{j}},
    \\[0.5em]
    \mathsf{C}_j &:= \eye{\Delta_j}\diamond\predictor_j,
    \qquad&&
    \mathsf{D}_j := (-\matop{U}_j)\diamond(\eye{\nabla_j}-\predictor_j \matop{U}_j),
\end{alignat*}
and $\eye{\Delta_j},\eye{\nabla_j}$ are the identity matrices of the appropriate size.

For the composite transform, we have
\begin{align}\label{eq:matmgard}
    \begin{bmatrix}
        \mathvec{\alpha}_0\\
        \mathvec{\beta}_0\\
        \vdots\\
        \mathvec{\beta}_{J-1}
    \end{bmatrix}
    =
    \mathsf{M}
    \cdot
    \mathvec{\alpha}_{J}
    =
    \prod_{j=0}^{J-1}
    \begin{bmatrix}
        \mathsf{M}_j
        &  \\
        & \eye{|\Delta_J|-|\Delta_{j}|}
    \end{bmatrix}
    \cdot
    \mathvec{\alpha}_{J},
\end{align}
and similarly for the inverse transform
\begin{align}\label{eq:matinvmgard}
    \mathvec{\alpha}_{J} 
    &=
    \mathsf{M}^{-1}
    \cdot
    \mathvec{\alpha}_{J}
    =
    \prod_{j=1}^{J}
    \begin{bmatrix}
        \mathsf{M}_{J-j}^{-1}
        &  \\
        & \eye{|\Delta_J|-|\Delta_{j}|}
    \end{bmatrix}
    \begin{bmatrix}
        \mathvec{\alpha}_0\\
        \mathvec{\beta}_0\\
        \vdots\\
        \mathvec{\beta}_{J-1}
    \end{bmatrix},
\end{align}
where $\mathsf{M}$ is the transform matrix.

\subsection{Refinement equations}

\subsubsection{Primal basis}

Given the inverse transform $\mgard_{j\to j+1}^{-1}$, one immediately gets the refinement equations for the nodal and wavelet bases.
Specifically, the transform $\mgard_{j\to j+1}^{-1}(\mathvec{\delta}^{\scriptscriptstyle\Delta_j}_k,\mathvec{0})=\text{\textit{Merge}}(\mathvec{\delta}^{\scriptscriptstyle\Delta_j}_k,\predictor_j\mathvec{\delta}^{\scriptscriptstyle\Delta_j}_k)$ gives the refinement equation for the nodal basis
\begin{align}\label{eq:basis_refined_phi}
    \varphi_{j,k} &= \varphi_{j+1,k} + \sum_{p\in\nabla_{j}} [\predictor_j]_{p,k} \varphi_{j+1,p},
    \qquad
    k\in\Delta_j,
\end{align}
where $\mathvec{\delta}^{\scriptscriptstyle\Delta_j}_k$ is the standard basis vector on $\Delta_j$ with the unit element in the $k$-th position. 
Similarly, the refinement equation for the wavelet basis is given by $\mgard_{j+1\to j}^{-1}(\mathvec{0},\mathvec{\delta}^{\scriptscriptstyle\nabla_j}_k)=\text{\textit{Merge}}(-\matop{U}_j \mathvec{\delta}^{\scriptscriptstyle\nabla_j}_k,(\eye{\nabla_j}-\predictor_j\matop{U}_j) \mathvec{\delta}^{\scriptscriptstyle\nabla_j}_k)$
\begin{align}\label{eq:basis_refined_psi}
    \psi_{j,k} &= 
    \varphi_{j+1,k}
    -\sum_{p\in\Delta_j}[\matop{U}_j]_{p,k} \varphi_{j+1,p}
    -\sum_{p\in\nabla_j}[\predictor_j\matop{U}_j]_{p,k} \varphi_{j+1,p}
    \qquad
    k\in\nabla_j.
\end{align}
One can use \eqref{eq:basis_refined_phi} to simplify the above expression as
\begin{align*}
    \psi_{j,k} 
    &= \varphi_{j+1,k}
    -\sum_{p\in\Delta_j}[\matop{U}_j]_{p,k} \varphi_{j+1,p}
    -\sum_{n\in\Delta_j}[\matop{U}_j]_{n,k}(\varphi_{j,n} - \varphi_{j+1,n})
    \\
    &= \varphi_{j+1,k}
    -\sum_{p\in\Delta_j}[\matop{U}_j]_{p,k}\varphi_{j,p}.
\end{align*}

In the vector-matrix form, the refinement equations in \eqref{eq:basis_refined_phi}-\eqref{eq:basis_refined_psi} are written as
\begin{align}
    \label{eq:refinement_phi}
    \mathvec{\varphi}_{j} 
    &= \mathvec{\varphi}_{j+1} \cdot \mathsf{C}_j,
    \\[1em]
    \label{eq:refinement_psi}
    \mathvec{\psi}_{j} 
    &=\mathvec{\varphi}_{j+1} \cdot \mathsf{D}_j,
\end{align}
where $\mathvec{\varphi}_{j}:=\{\varphi_{j,k}:k\in\Delta_j\}$, $\mathvec{\psi}_{j}:=\{\psi_{j,k}:k\in\nabla_j\}$ are the row vectors of basis functions. 
The refinement equation \eqref{eq:refinement_phi} is somewhat redundant as the nodal basis at each level is fixed a priory.

Recurrent application of the above formulas gives the so-called cascade algorithm that implicitly defines the basis functions at any point in the domain.
Specifically, equations \eqref{eq:refinement_phi}-\eqref{eq:refinement_psi} involve the same matrices as the matrix form of the inverse transform. 
The nodal representation of the basis vectors at any level $J$ is then given by the columns of the matrices in \eqref{eq:matinvmgard}~as
\begin{align}
    \label{eq:refined_phi}
    \mathvec{\varphi}_{j} 
    &=\mathvec{\varphi}_{J} \cdot \left( \prod_{l=j+1}^{J-1} \mathsf{C}_{J+j-l} \right)
    \cdot \mathsf{C}_j,
    \\
    \label{eq:refined_psi}
    \mathvec{\psi}_{j} 
    &=\mathvec{\varphi}_{J} \cdot \left( \prod_{l=j+1}^{J-1} \mathsf{C}_{J+j-l} \right)
    \cdot \mathsf{D}_j.
\end{align}

\subsubsection{Dual basis}
In order to determine the values of the dual basis functions $\tilde{\mathvec{\varphi}}$, $\tilde{\mathvec{\psi}}$, one can use the definition of the wavelet coefficients in the expansion
\begin{align*}
    P_ju = \sum_{k\in\Delta_j} \alpha_{j,k} \varphi_{j,k} + \sum_{l\geq j}\sum_{k\in\nabla_l} \beta_{l,k}\psi_{l,k}
    = \sum_{k\in\Delta_j} \langle u,\tilde{\varphi}_{j,k}\rangle \varphi_{j,k} + \sum_{l\geq j}\sum_{k\in\nabla_l} \langle u,\tilde{\psi}_{l,k}\rangle\psi_{l,k}.
\end{align*}
When $u:=\delta(\cdot-y)$ is the Dirac measure centered at $y$, we get $\tilde{\psi}_{j,k}(y):=\langle\delta(\cdot-y),\tilde{\psi}_{j,k}\rangle$.
Hence, the coefficients of the above expansion provide the values of the dual basis functions at a given point.
These values are given by the forward MGARD transform $(\tilde{\mathvec{\varphi}}_j,\tilde{\mathvec{\psi}}_j)_p=\mgard_{J\to j}(\varepsilon_{J,p}^{-1}\mathvec{\delta}_p^{\scriptscriptstyle\Delta_J})$.
The scaling $\varepsilon_{J,p}^{-1}$ of the standard basis vector $\mathvec{\delta}_p^{\scriptscriptstyle\Delta_J}$ ensures the unit integral of the approximate delta function at the level $J$ and is required for the weak convergence to the Dirac distribution when $J\to\infty$.
The nodal representation of the dual basis vectors $\tilde{\mathvec{\varphi}},\tilde{\mathvec{\psi}}$ at the nodes of $\Delta_J$ is thus given by 
\begin{align}
    \label{eq:refined_dualphi}
    \tilde{\mathvec{\varphi}}_{j} 
    &= \mathsf{A}_j \cdot
    \Bigg( \prod_{l=j+1}^{J-1} \mathsf{A}_l \Bigg) \cdot \mathvec{\varepsilon}_{J}^{-1} \cdot \mathvec{\varphi}_{J}^T,
    \\
    \label{eq:refined_dualpsi}
    \tilde{\mathvec{\psi}}_{j} 
    &= \mathsf{B}_j \cdot
    \Bigg( \prod_{l=j+1}^{J-1} \mathsf{A}_l \Bigg) \cdot \mathvec{\varepsilon}_{J}^{-1} \cdot \mathvec{\varphi}_{J}^T,
\end{align}
where $\tilde{\mathvec{\varphi}}_{j}:=\{\tilde{\varphi}_{j,k}:k\in\Delta_j\}$, $\tilde{\mathvec{\psi}}_{j}:=\{\tilde{\psi}_{j,k}:k\in\nabla_j\}$ are the column vectors of dual basis functions, and $\mathvec{\varepsilon}_{J}^{-1} := \diag\big(\{\varepsilon_{J,k}^{-1}:k\in\Delta_J\}\big)$ is the diagonal scaling matrix.

By comparing the expressions for $\tilde{\mathvec{\varphi}}_{j}$, $\tilde{\mathvec{\varphi}}_{j-1}$, one gets the refinement equations for the dual basis as well
\begin{align}
    \label{eq:refinement_dualphi}
    \tilde{\mathvec{\varphi}}_{j} 
    &= \mathsf{A}_j \cdot \tilde{\mathvec{\varphi}}_{j+1},
    \\[0.5em]
    \label{eq:refinement_dualpsi}
    \tilde{\mathvec{\psi}}_{j} 
    &= \mathsf{B}_j \cdot \tilde{\mathvec{\varphi}}_{j+1},
\end{align}
with
\begin{align}
    \tilde{\mathvec{\varphi}}_{J-1} 
    &= \mathsf{A}_{J-1} \cdot \mathvec{\varepsilon}_{J}^{-1} \cdot \mathvec{\varphi}_{J}^T,
    \\[0.5em]
    \tilde{\mathvec{\psi}}_{J-1} 
    &= \mathsf{B}_{J-1} \cdot \mathvec{\varepsilon}_{J}^{-1} \cdot \mathvec{\varphi}_{J}^T.
\end{align}

Similarly to \eqref{eq:basis_refined_phi}-\eqref{eq:basis_refined_psi}, we can write the refinement equations for each basis vector~as
\begin{align}\label{eq:basis_refined_dualphi}
    \tilde{\psi}_{j,k}
    &= \tilde{\varphi}_{j+1,k} - \sum_{p\in\Delta_j} [\predictor_{j}]_{k,p} \tilde{\varphi}_{j+1,p},
    \quad k\in\nabla_j,
    \\
    \label{eq:basis_refined_dualpsi}
    \tilde{\varphi}_{j,k}
    &= \tilde{\varphi}_{j+1,k} + \sum_{p\in\nabla_j} [\matop{U}_{j}]_{k,p} \tilde{\psi}_{j,p},\phantom{_{+1}}
    \quad k\in\Delta_j.
\end{align}

\section{Uniform dyadic grids on $\Omega:=[0,1]$}
\label{sec:uniform_grids}

\begin{figure}[tb]
    \centering
    \def\w{\textwidth}
    \stackinset{c}{0pt}{t}{-12pt}{}{\includegraphics[width=\w]{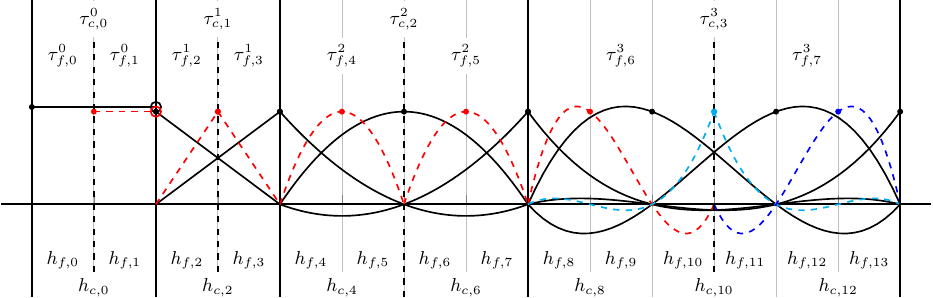}}
    \caption{One-dimensional coarse (solid) and surplus (dashed) Lagrange basis functions at two adjacent grid levels $\mathcal{G}_{c/f}$. The corresponding one-dimensional elements $\tau_{c/f}^q$ and grid spacings $h_{c/f}$ are also shown.}
    \label{fig:inter_level_basis}
\end{figure}

The above results are valid for any choice of the nodal basis for the nested spaces $V_j\subset V_{j+1}$.
Here we focus on the particular case of the piecewise polynomial basis on dyadically refined uniform grids of the unit interval in one dimension.

Without loss of generality, consider a two-level decomposition scheme for the coarse and fine grids $\mathcal{G}_j\subset\mathcal{G}_{j+1}$.
Define a 1D element $\tau\in\mathcal{T}_{j}$ of order $q>0$ as a collection of $q+1$ consecutive nodes of the corresponding grid.
The element of order $q=0$ is defined using the c{\'a}dl{\'a}g (right continuous with left limit) nodal basis function, see Figure~\ref{fig:inter_level_basis} for the illustration.
Each grid $\mathcal{G}_j$ has $2^j$ elements, and we assume that $\mathcal{G}_0$ has $q+1$ nodes.
We also use the primed indices to denote the local-to-the-element numbering of nodes, e.g., $p'\in[0,\hdots,q] \leftrightarrow p\in\Delta_j^\tau$.
Finally, $\tau_k$ denotes the element that contains the node with index $k$.

In this case, one has
\begin{align*}
    \predictor_j = \underbrace{\predictor^q\overt\hdots\overt\predictor^q}_{2^j\text{ times}},
\end{align*}
with the $q$-order interpolation matrix $\predictor^q\in\mathbb{R}^{q\times(q+1)}$ defined in Appendix~\ref{app:interpmat}. 
The refinement equations involving the interpolation matrix then simplify to
\begin{align*}
    \varphi_{j,k} 
    &= \varphi_{j+1,k} + \sum_{p\in\nabla_{j}} [\predictor_j]_{p,k} \varphi_{j+1,p}
    =  \varphi_{j+1,k} + \sum_{p\in\nabla_{j}^{\tau_k}} [\predictor^q]_{p',k'} \varphi_{j+1,p}
    \quad
    k\in\Delta_j,
    \\
    \tilde{\psi}_{j,k} 
    &= \tilde{\varphi}_{j+1,k} - \sum_{p\in\Delta_j} [\predictor_j]_{k,p} \tilde{\varphi}_{j+1,p}
    = \tilde{\varphi}_{j+1,k} - \sum_{p\in\Delta_j^{\tau_k}} [\predictor^q]_{k',p'} \tilde{\varphi}_{j+1,p},
    \quad
    k\in\nabla_{j}.
\end{align*}

\subsubsection{Normalization}

The $L^2$ norms of the one-dimensional basis functions can be calculated exactly using the Newton-Cotes quadratures since each $\varphi_{j,k}$ is a compactly supported piecewise polynomial. 
This gives
\begin{align}
    \|\varphi_{j,k}\|^2 
    &:= \sum_{\tau\in\mathcal{T}_j} \int_{\tau} \varphi_{j,k}^2 dx 
    = 2^{-j} \sum_{\tau\in\mathcal{T}_j} \sum_{p\in\Delta_{j+1}^{\tau}} w_{p'} \big[\mathsf{C}_j]_{p,k}^2,
    \\
    \|\psi_{j,k}\|^2 
    &:= \sum_{\tau\in\mathcal{T}_j} \int_{\tau} \psi_{j,k}^2 dx 
    = 2^{-j} \sum_{\tau\in\mathcal{T}_j} \sum_{p\in\Delta_{j+1}^{\tau}} w_{p'} \big[\mathsf{D}_j]_{p,k}^2,
\end{align}
where $2^{-j}$ is the length of the element of the grid $\mathcal{G}_j$.
The quadrature weights $w_{p'}$ can be found, e.g., in \cite{abramowitz1988handbook,Sermutlu_2022}.
Since each $\varphi_{j,k}$ is supported on at most two elements, for the indices $k$ inside the element, we have
\begin{align}
    \|\varphi_{j,k}\|^2 
    = \int_{\tau_{k}} \varphi_{j,k}^2 dx 
    = 2^{-j} \Big( w_{k'} + \sum_{p'=0}^{q-1} w_{2p'+1} \big[\predictor^q]_{p',k'}^2 \Big).
\end{align}
Similarly, for the boundary nodes shared between the elements, we get
\begin{align}
    \|\varphi_{j,k'}\|^2 
    = 2^{1-j} \Big( w_{0} + \sum_{p'=0}^{q-1} w_{2p'+1} \big[\predictor^q]_{p',k'}^2 \Big).
\end{align}

\subsubsection{Interpolating projectors}

\begin{figure}[t]
	\centering
    \def\h{0.21\textwidth}
    \begin{flushright}
    \includegraphics[width=0.28\textwidth]{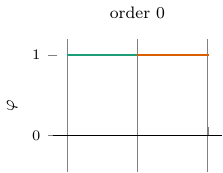}\;
    \includegraphics[width=\h]{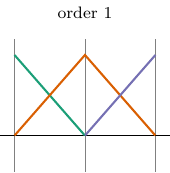}\;
    \includegraphics[width=\h]{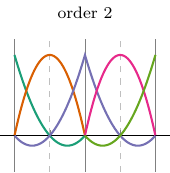}\;
    \includegraphics[width=\h]{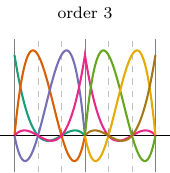}
    \\[1em]
    \includegraphics[width=0.3\textwidth]{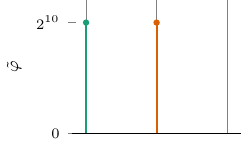}\;
    \includegraphics[width=\h]{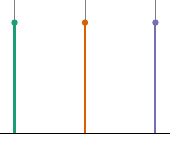}\;
    \includegraphics[width=\h]{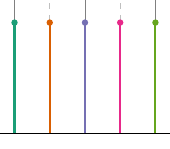}\;
    \includegraphics[width=\h]{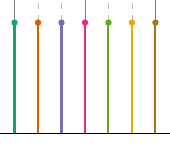}
    \\[1em]
    \includegraphics[width=0.28\textwidth]{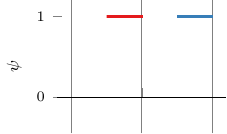}\;
    \includegraphics[width=\h]{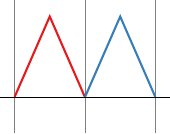}\;
    \includegraphics[width=\h]{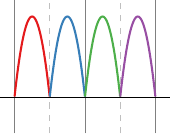}\;
    \includegraphics[width=\h]{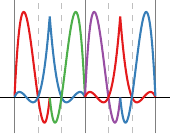}
    \\[1em]
    \includegraphics[width=0.31\textwidth]{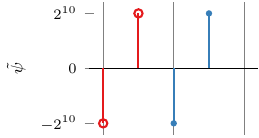}\;
    \includegraphics[width=\h]{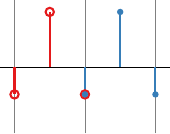}\;
    \includegraphics[width=\h]{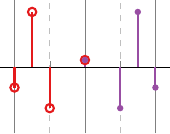}\;
    \includegraphics[width=\h]{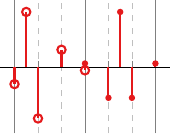}
    \end{flushright}
     \caption{The primal and dual basis functions induced by $\mathcal{P}_j=\mathcal{I}_j$ at level $1$.
     For clarity, only one basis vector per element is shown for the dual wavelet basis~$\tilde{\psi}$.
     }
     \label{fig:interpolation_basis} 
\end{figure}

For the case of interpolating basis with $\mathcal{P}_j=\mathcal{I}_j$, one gets $\matop{U}_j = \mathvec{0}$ which simplifies the refinement equations \eqref{eq:basis_refined_phi}-\eqref{eq:basis_refined_psi}, \eqref{eq:basis_refined_dualphi}-\eqref{eq:basis_refined_dualpsi} yielding
\begin{align*}
    \mathvec{\varphi}_{j} 
    &= \Big\{ \varphi_{j,k}: k\in\Delta_j \Big\},
    \qquad
    \mathvec{\psi}_{j} 
    = \Big\{ \varphi_{j+1,k} : k\in\nabla_{j} \Big\}.
\end{align*}
So the approximation basis is just the nodal basis of the grid at level $j$
while the wavelet basis consists of the next-level nodal basis vectors at the surplus nodes in~$\nabla_j$.

For the dual bases, one gets
\begin{align*}
    \tilde{\mathvec{\varphi}}_{j} 
    &= \Big\{ \tilde{\varphi}_{j+1,k}: k\in\Delta_j \Big\}
    = \Big\{ 2^J \varphi_{J,k}: k\in\Delta_j \Big\}
    \xrightarrow{J\to\infty} \Big\{ \delta(\cdot-x_k): k\in\Delta_j \Big\},
    \\
    \tilde{\mathvec{\psi}}_{j} 
    &= \Big\{ \tilde{\varphi}_{j+1,k} - \sum_{p\in\Delta_j^{\tau_k}} [\predictor^q]_{k',p'} \tilde{\varphi}_{j+1,p} : k\in\nabla_{j} \Big\}
    \\
    &\xrightarrow{J\to\infty} \Big\{ \delta(\cdot-x_k) - \sum_{p\in\Delta_j^{\tau_{k}}} [\predictor^q]_{k',p'} \delta\big(\cdot-x_{p}\big) : k\in\nabla_{j} \Big\}.
\end{align*}
Hence, the dual approximation basis vectors are the Dirac measures centered at the nodes of $\Delta_j$.
The dual wavelet vectors have a Dirac measure at the corresponding node of $\nabla_j$ and $q+1$ additional Diracs at the nodes of $\Delta_j^{\tau_k}$ inside the corresponding element $\tau_k$, see Figure~\ref{fig:interpolation_basis}.

The normalization of the wavelet basis simplifies to
\begin{align}
    \|\psi_{j,k}\|^2 
    &= \frac{q}{2^j} \sum_{\tau_k\in\mathcal{T}_j} \sum_{p\in\Delta_{j+1}^{\tau_k}} w_{p'} \big[\matop{0}\diamond\eye{\nabla_j}]_{p,k}^2
    = \frac{q}{2^j} \sum_{p'=0}^{q} w_{2p'+1}.
\end{align}

\subsubsection{Piecewise constant predictor}

In this case, one has $\predictor_j:=\eye{}$, $\matop{U}_j:=\frac{1}{2}\eye{}$, and $\varepsilon_{J,k} = 2^{-J}$.
This gives
\begin{align*}
    \mathvec{\varphi}_{j} 
    &=\{ \varphi_{j,k}: k\in\Delta_j\},
    \qquad
    \mathvec{\psi}_{j} 
    = \left\{ \frac{1}{2} ( \varphi_{j+1,k} - \varphi_{j+1,k-1} ) : k\in\nabla_j \right\},
\end{align*}
and
\begin{align*}
    \tilde{\mathvec{\varphi}}_{j} 
    &= \left\{ \frac{1}{2} (\tilde{\varphi}_{j+1,k}+\tilde{\varphi}_{j+1,k+1}) : k\in\Delta_j\right\}
    = \left\{ 2^{j-J} \smashoperator{\sum_{p\in\Delta_{j}^{\tau_{k}}}} \tilde{\varphi}_{J,p} : k\in\Delta_j\right\},
    \\[0.5em]
    \tilde{\mathvec{\psi}}_{j} 
    &= \left\{ \tilde{\varphi}_{j+1,k}-\tilde{\varphi}_{j+1,k-1} : k\in\nabla_j\right\}
    = \left\{ 2^{j+1-J} \Big( {\sum_{p\in\Delta_{j+1}^{\tau_{k}}}} \tilde{\varphi}_{J,p} - \smashoperator{\sum_{p\in\Delta_{j+1}^{\tau_{k-1}}}} \tilde{\varphi}_{J,p} \Big) : k\in\nabla_j\right\}.
\end{align*}
By noting that $\smashoperator{\sum_{p\in\Delta_{j}^{\tau_{k}}}} \tilde{\varphi}_{J,p} = 2^J \smashoperator{\sum_{p\in\Delta_{j}^{\tau_{k}}}} \varphi_{J,p} = 2^J \varphi_{j,k}$, we get
\begin{align*}
    \tilde{\mathvec{\varphi}}_{j}  = \left\{ 2^{j} \varphi_{j,k} : k\in\Delta_j\right\},
    \qquad
    \tilde{\mathvec{\psi}}_{j}  = \left\{ 2^{j+1} (\varphi_{j+1,k}-\varphi_{j+1,k-1}) : k\in\nabla_j\right\}.
\end{align*}
The normalization of the basis is trivial and is given by
\begin{align*}
    \|\varphi_{j,k}\|^2 = \frac{1}{2^j},
    \qquad
    \|\psi_{j,k}\|^2 = \frac{1}{2^{j+2}}.
\end{align*}
One can verify that
\begin{align*}
    \langle \varphi_{j,k}, \tilde{\varphi}_{j,k'} \rangle &= \langle \varphi_{j,k}, \varphi_{j,k'} \rangle = 
    \langle \psi_{j,k}, \tilde{\psi}_{j,k'} \rangle = \langle \psi_{j,k}, \psi_{j,k'} \rangle = \delta_{k,k'},
\end{align*}
i.e, this is the classical orthogonal Haar basis, see Figures~\ref{fig:projection_basis}-\ref{fig:dg_basis}.

\begin{figure}[t]
	\centering
    \def\h{0.21\textwidth}
    \begin{flushright}
    \includegraphics[width=0.3\textwidth]{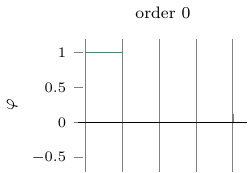}\;
    \includegraphics[width=\h]{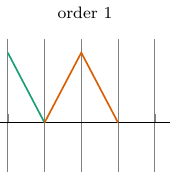}\;
    \includegraphics[width=\h]{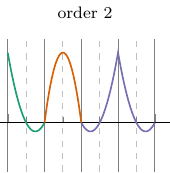}\;
    \includegraphics[width=\h]{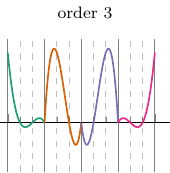}
    \\[1em]
    \includegraphics[width=0.29\textwidth]{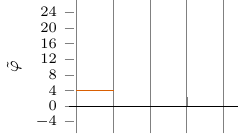}\;
    \includegraphics[width=\h]{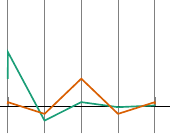}\;
    \includegraphics[width=\h]{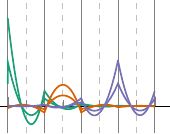}\;
    \includegraphics[width=\h]{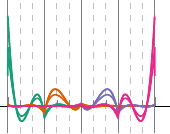}
    \\[1em]
    \includegraphics[width=0.29\textwidth]{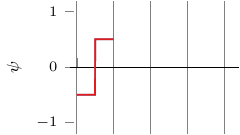}\;
    \includegraphics[width=\h]{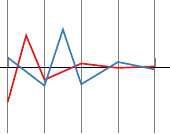}\;
    \includegraphics[width=\h]{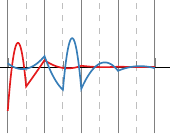}\;
    \includegraphics[width=\h]{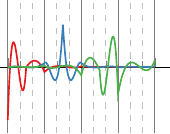}
    \\[1em]
    \includegraphics[width=0.31\textwidth]{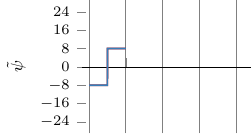}\;
    \includegraphics[width=\h]{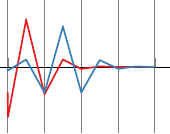}\;
    \includegraphics[width=\h]{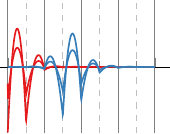}\;
    \includegraphics[width=\h]{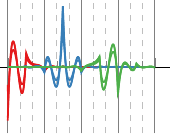}
    \end{flushright}
    \caption{The primal and dual basis functions induced by $\mathcal{P}_j^{cg}$ at level $2$. 
    For clarity, only one basis vector is shown in each element.
    }
    \label{fig:projection_basis} 
\end{figure}

\begin{figure}[t]
	\centering
    \def\h{0.21\textwidth}
    \begin{flushright}
    \includegraphics[width=0.31\textwidth]{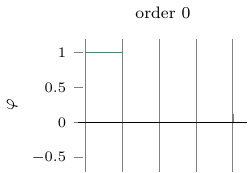}\;
    \includegraphics[width=\h]{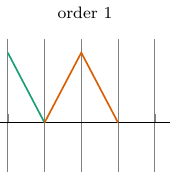}\;
    \includegraphics[width=\h]{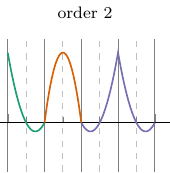}\;
    \includegraphics[width=\h]{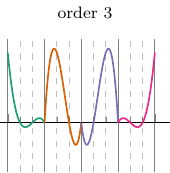}
    \\[1em]
    \includegraphics[width=0.31\textwidth]{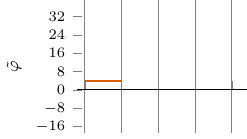}\;
    \includegraphics[width=\h]{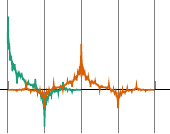}\;
    \includegraphics[width=\h]{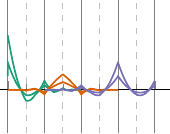}\;
    \includegraphics[width=\h]{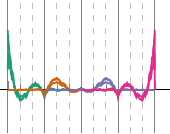}
    \\[1em]
    \includegraphics[width=0.31\textwidth]{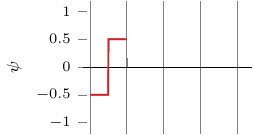}\;
    \includegraphics[width=\h]{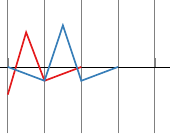}\;
    \includegraphics[width=\h]{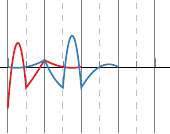}\;
    \includegraphics[width=\h]{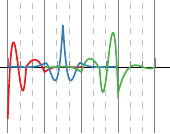}
    \\[1em]
    \includegraphics[width=0.31\textwidth]{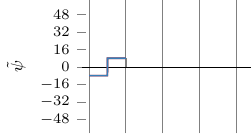}\;
    \includegraphics[width=\h]{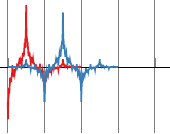}\;
    \includegraphics[width=\h]{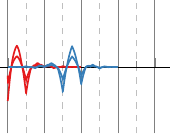}\;
    \includegraphics[width=\h]{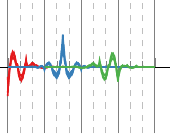}
    \end{flushright}
    \caption{The primal and dual basis functions induced by $\mathcal{P}_j^{dg}$ at level $2$.
    For clarity, only one basis vector is shown in each element.
    }
    \label{fig:dg_basis} 
\end{figure}

\subsubsection{Piecewise polynomial predictors with stable projectors.}

The local Gramm matrices $\matop{G}^{q}_{\scriptscriptstyle\Delta\Delta}$, $\matop{G}^{q}_{\scriptscriptstyle\Delta\nabla}$ of a varying order $q$ are defined in Appendix~\ref{app:grammmat}.
The update matrix spanning $r$ elements is then given by 
\begin{align*}
    \Big(\underbrace{\matop{G}^{q}_{\scriptscriptstyle\Delta\Delta} \oplus \hdots \oplus \matop{G}^{q}_{\scriptscriptstyle\Delta\Delta}}_{r\text{ times}}\Big)^{-1} 
    \Big(\underbrace{\matop{G}^{q}_{\scriptscriptstyle\Delta\nabla} \ominus \hdots \ominus \matop{G}^{q}_{\scriptscriptstyle\Delta\nabla}}_{r\text{ times}}\Big).
\end{align*}
Therefore, the update matrix induced by the global projectors $\mathcal{P}_j^{cg}$ is given by
\begin{align}\label{eq:Ucg}
    \matop{U}_j^{cg} = 
    \Big(\underbrace{\matop{G}^{q}_{\scriptscriptstyle\Delta\Delta} \oplus \hdots \oplus \matop{G}^{q}_{\scriptscriptstyle\Delta\Delta}}_{2^j\text{ times}}\Big)^{-1} 
    \Big(\underbrace{\matop{G}^{q}_{\scriptscriptstyle\Delta\nabla} \ominus \hdots \ominus \matop{G}^{q}_{\scriptscriptstyle\Delta\nabla}}_{2^j\text{ times}}\Big),
\end{align}
while the update matrix induced by the element-wise projectors $\mathcal{P}_{j}^{dg}$ is given by
\begin{align}\label{eq:Udg}
    \matop{U}_j^{dg} 
    =
    \Big( \underbrace{\eye{q+1}\oplus\hdots\oplus\eye{q+1}}_{2^j\text{ times}} \Big)^{-1}
    \Big( \underbrace{\matop{U}^q\ominus\hdots\ominus\matop{U}^q}_{2^j\text{ times}} \Big).
\end{align}
To clarify the meaning of the expression for $\matop{U}_j^{dg}$, consider the case with $q=1$ and $j=1$, which gives
\begin{align*}
    \Big( \eye{2}\oplus\eye{2}  \Big)^{-1}
    \Big( \matop{U}^2\ominus\matop{U}^2 \Big)
    &= 
    \begin{bmatrix}
        1 &   & \\
          & 2 & \\
          &   & 1
    \end{bmatrix}^{-1}
    \begin{bmatrix}
        u_{00} & u_{01} &        &        \\
        u_{10} & u_{11} & u_{00} & u_{01} \\
               &        & u_{10} & u_{11} \\
    \end{bmatrix}
    \\
    &=
    \begin{bmatrix}
        u_{00} & u_{01} &        &        \\
        \frac{u_{10}}{2} & \frac{u_{11}}{2} & \frac{u_{00}}{2} & \frac{u_{01}}{2} \\
               &        & u_{10} & u_{11} \\
    \end{bmatrix}.
\end{align*}
It is worth noting that the global update matrix $\matop{U}_j^{cg}$ is dense.
Moreover, its evaluation requires inverting the banded matrix which adds to the computational complexity and scales with the problem size.
On the other side, the matrix $\matop{U}^q\ominus\hdots\ominus\matop{U}^q$ in $\matop{U}_j^{dg}$ is comprised of identical blocks which can be explicitly precomputed. 
Figure~\ref{fig:update_scaling} shows the costs of assembling and solving the linear systems required to evaluate the update matrices $\matop{U}^{cg}$ and $\matop{U}^{dg}$ in \eqref{eq:Ucg}-\eqref{eq:Udg}.
Since evaluation of $\matop{U}^{dg}$ on the uniform grid requires matrix inversion just in one element, its cost is negligible and is omitted.
At the same time, the cost of a banded linear solver in computing $\matop{U}^{cg}$ scales super-linearly with the problem size.
The costs of assembling the matrices are also non-negligible and scale linearly.
However, $\matop{U}^{cg}$ requires assembling two matrices while $\matop{U}^{dg}$ needs only one.

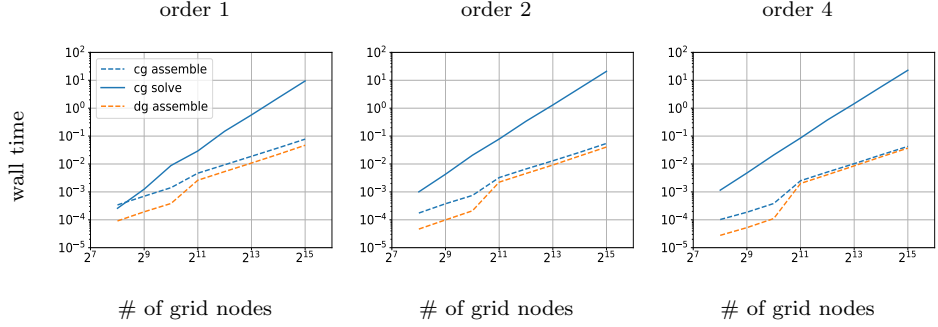
\begin{figure}[t]
    \centering
    \pgfplotsset{every tick label/.append style={font=\scriptsize}}
    \pgfplotsset{every axis label/.append style={font=\footnotesize}}
    \pgfplotsset{every axis title/.append style={font=\footnotesize}} 	
    
    \foreach \order in {1,2,4}
    {
        \begin{tikzpicture}[trim axis left,trim axis right]
            \begin{axis} [
                width=0.42\textwidth,
                ymode = normal,
                xticklabels={,,},
                yticklabels={,,},
                enlargelimits=false,
                axis on top,
                axis line style={draw=none},
                title={order $\order$},
                xlabel=\# of grid nodes,
                ylabel=\ifthenelse{\equal{\order}{1}}{wall time}{},
                tick style={draw=none},
                very thick ]
                \addplot graphics [points={(0,0) (1,1)}, includegraphics={trim=0 0 0 0,clip}] {{{images/assemble_solve_vs_dof_\order}.pdf}};
            \end{axis}
        \end{tikzpicture}
    }
    \caption{The costs of assembling and solving the linear systems in \eqref{eq:Ucg}-\eqref{eq:Udg} averaged over $100$ runs..}
    \label{fig:update_scaling}
\end{figure}

In either of two cases, the explicit representation of the basis functions is challenging.
Figures~\ref{fig:projection_basis}-\ref{fig:dg_basis} show the basis functions induced by $\mathcal{P}^{cg}$ and $\mathcal{P}^{dg}$ generated implicitly using the refinement equations in \eqref{eq:refinement_phi}-\eqref{eq:refinement_psi}, \eqref{eq:refinement_dualphi}-\eqref{eq:refinement_dualpsi}.
One can see that $\mathcal{P}^{cg}$ yield bases with global support while $\mathcal{P}^{dg}$ induced bases are supported only at the neighboring elements.
Another observation is the fact that the dual basis functions are given by the superposition of $q$ different functions corresponding to different nodes of the elements. 
Particularly, the continuous functions in Figures~\ref{fig:projection_basis}-\ref{fig:dg_basis} are obtained by connecting the values at the matching nodes of each element.
Also note the the jump discontinuities at the boundaries of the interval.

\subsection{Tensor grids}

Extending the one-dimensional construction to the case of tensor grids is almost trivial. 
Given the matrix $\mathsf{M}$ of the one-dimensional transform in \eqref{eq:matmgard}-\eqref{eq:matinvmgard}, one can simply use the Kronecker product as follows

\begin{align*}
    \tilde{\mathvec{\alpha}}_0
    &=
    \bigotimes_{d=1}^D
    \mathsf{M}
    \cdot
    \mathvec{\alpha}_J,
    \qquad
    \mathvec{\alpha}_J
    =
    \bigotimes_{d=1}^D
    \mathsf{M}^{-1}
    \cdot
    \tilde{\mathvec{\alpha}}_0.
\end{align*}
Alternatively, one can take the tensor product of the matrices in \eqref{eq:levelmatmgard}, and apply these transforms level-by-level.
In wavelet methods, the first approach is referred to as fully separable transform, while the second approach, proposed by Mallat in \cite{Mallat1989}, is more commonly employed in practice.

\section{Results}
\label{sec:results}

In this section, we discuss the use of piecewise polynomials in wavelet constructions and compare this approach to traditional translation-invariant filter banks. 
We evaluate the performance of these methods in terms of coefficient decay for different sampling strategies. 
Next, we explore the impact of projection operator order on data compression, analyzing how it affects the compression ratio and reconstruction quality. 
These comparisons highlight the trade-offs involved in selecting the appropriate approach for different data types.

\subsection{Piecewise polynomials.}
The proposed wavelet construction utilizes a hierarchical decomposition of piecewise polynomial spaces, which contrasts with the traditional translation-invariant filter bank formulation. 
In this example, we compare these two approaches by considering the function
\begin{align*}
    f(x) = \sin(2\pi x) + \frac{1}{3} \sin(11\pi x) + \frac{1}{5} \sin(23\pi x),
    \qquad
    x\in[0,1],
\end{align*}
and its piecewise linear interpolant $\mathcal{I}f$ defined on a grid with $N_e = 2^{10}$ elements. 
The function $\mathcal{I}f$ is linear within each element and continuous across elements. 
Figure~\ref{fig:projections} illustrates this on a coarser grid.

We evaluate this function at $N_p$ nodes of the uniform grid over the unit interval and compare the decay of the transform coefficients using Daubechies orthogonal wavelets (`db2') and three MGARD transforms with different projection operators. 
All considered transforms have two vanishing moments, meaning their wavelet bases are orthogonal to polynomials of orders zero and one. 
Additionally, MGARD wavelets are orthogonal to piecewise polynomials by design.

Figure~\ref{fig:ex1_coef_decay} highlights these distinctions. 
It shows that while all transforms yield similar results when applied to data evaluated at the element nodes, the MGARD transform coefficients decay much faster when the data is sampled at multiple nodes within the elements. 
In contrast, Daubechies wavelets produce vanishing coefficients only when the wavelet is entirely contained within an element. 
The lack of $C^1$ smoothness across elements in the Daubechies wavelets results in larger coefficients that propagate through the hierarchy.

\begin{figure}[tb]
    \centering
    \def\w{0.28\textwidth}
    \stackinset{c}{0pt}{t}{-12pt}{\small $\mathcal{I}f(x)$}{\includegraphics[width=\w]{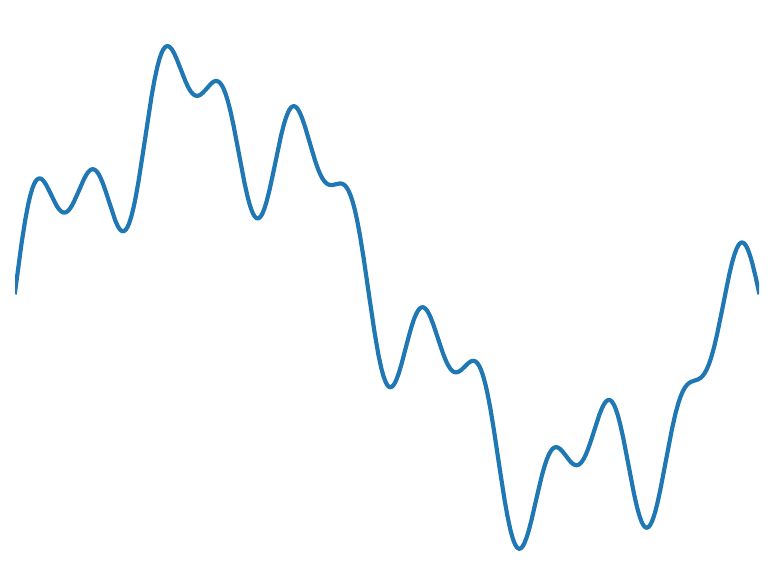}}
    \qquad
    \stackinset{c}{0pt}{t}{-12pt}{\small $N_e=2^{10}$, $N_p=2^{10}$}{\includegraphics[width=\w]{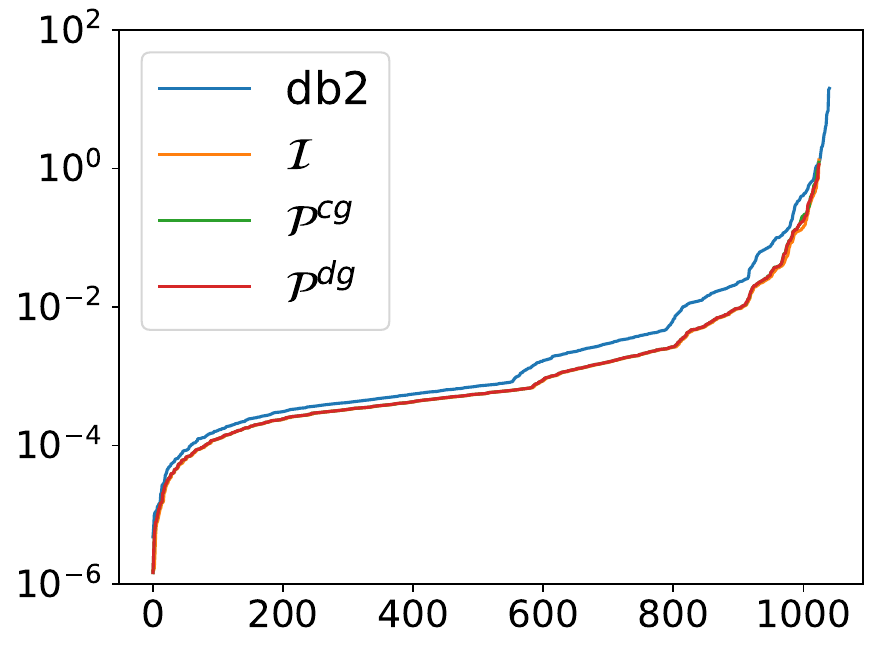}}
    \qquad
    \stackinset{c}{0pt}{t}{-12pt}{\small $N_e=2^{10}$, $N_p=2^{12}$}{\includegraphics[width=\w]{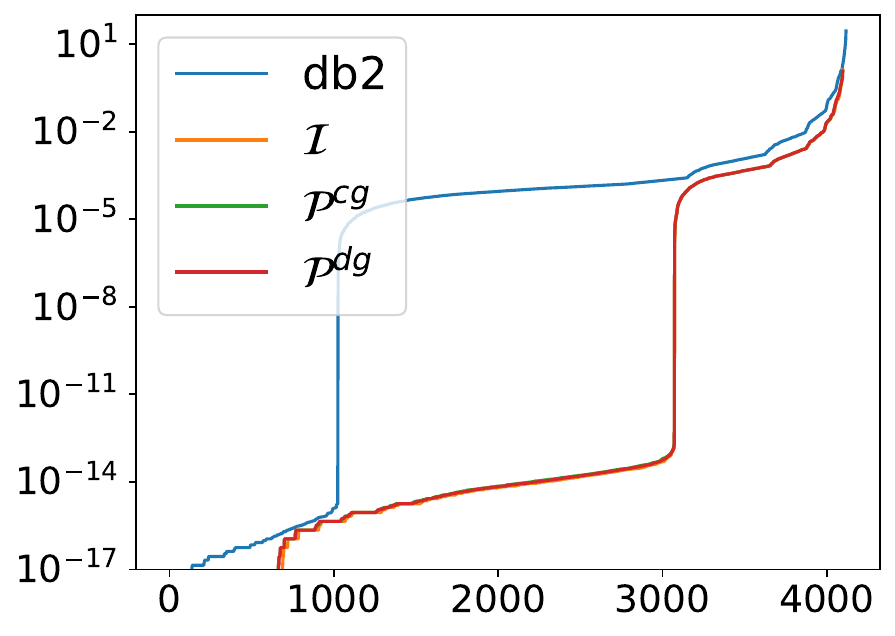}}
    \caption{Sorted coefficients of the wavelet transforms with two vanishing moments for the piecewise linear function $\mathcal{I}f(x)$ defined on $N_e$ elements and evaluated at $N_p$ nodes. }
    \label{fig:ex1_coef_decay}
\end{figure}

\subsection{Order comparison.}

In this example, we study the impact of the order on data compression.
As a measure of the compression quality, we use the $L^2$ error of the reconstructed signal as a function of the compression ratio (CR).
Here we define the compression ratio as the relative number of the transform coefficients above the given threshold, and we do not consider quantization or entropy coding.

Figure~\ref{fig:ex2_image} shows the coefficient decays and the compression ratios of the MGARD transform induced by global projectors $\mathcal{P}^{cg}$ of different orders.
The top row depicts results for an image from the USC-SIPI image database\footnote{\url{https://sipi.usc.edu/database/}}.
It can be observed that all methods perform nearly identically, with lower-order predictors performing the best. 
This is due to the limited smoothness of natural images, as explained in \cite{DeVore1992_smoothness, DeVore1992_smoothness_1994}. 
Therefore, lower-order methods are expected to yield better results.

The second row of Figure~\ref{fig:ex2_image} presents data from the boundary element (BEM) approximation of the solution to the Helmholtz equation describing scattering from a sphere. 
This data was generated using the boundary element software package in \cite{betcke2021bempp}. 
BEM solutions to the Laplace and Helmholtz equations are known to be smooth away from the boundary. 
This is evident from the decay of the wavelet coefficients for this problem. 
As a result, higher-order methods are much superior in this case, as indicated by the observed compression ratios.

\begin{figure}[tb]
    \centering
    \def\w{0.23\textwidth}
    \stackinset{c}{0pt}{t}{-12pt}{\small data}{\includegraphics[height=\w]{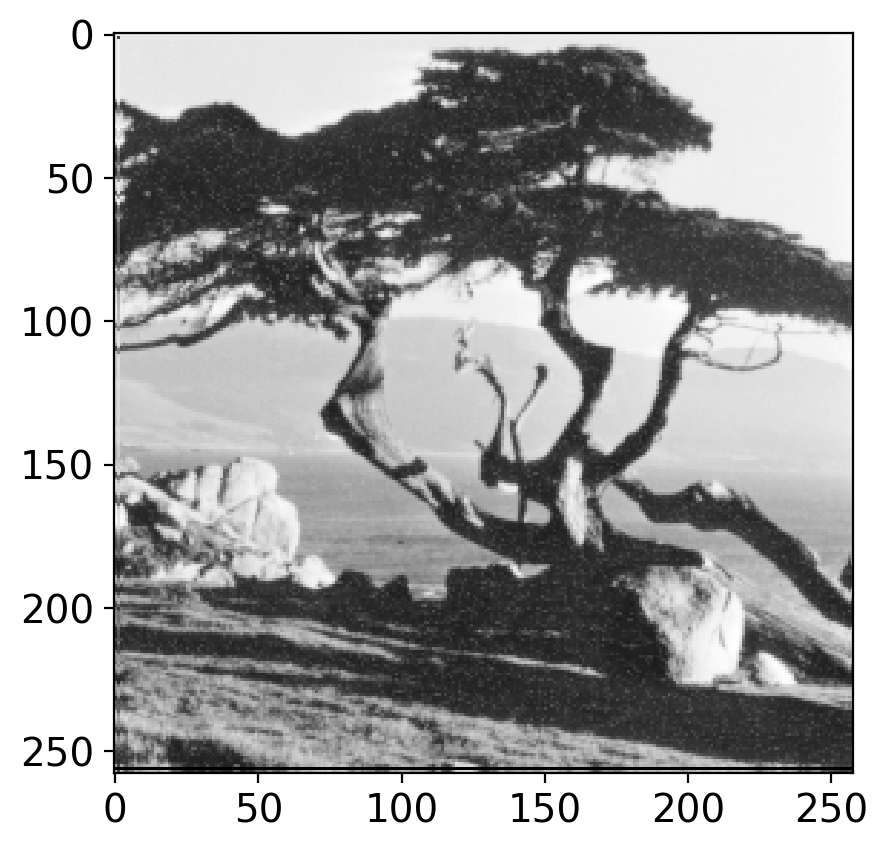}}
    \quad
    \stackinset{c}{0pt}{t}{-12pt}{\small coefficient decay}{\includegraphics[height=\w]{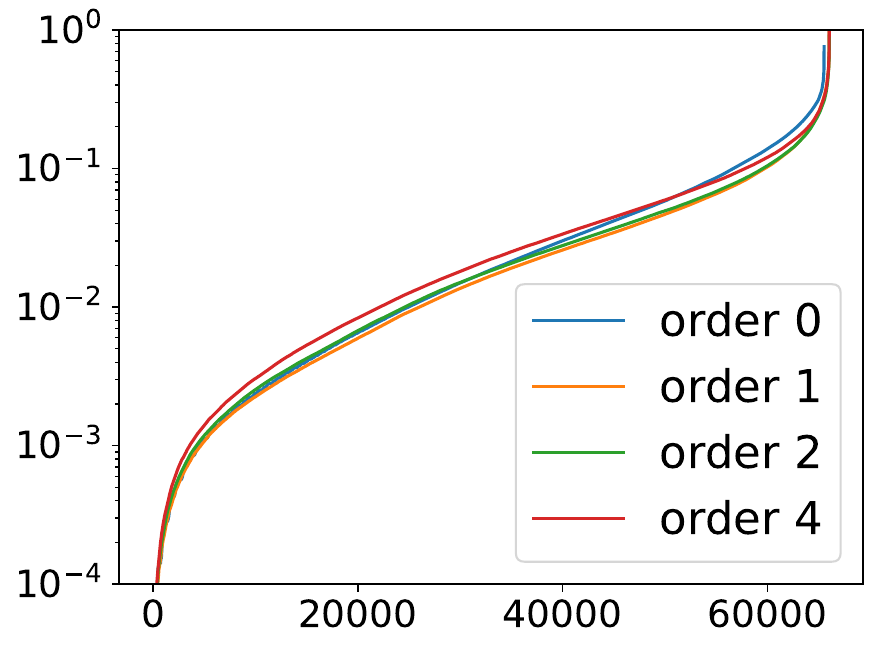}}
    \quad
    \stackinset{c}{0pt}{t}{-12pt}{\small compression ratio vs error}{\includegraphics[height=\w]{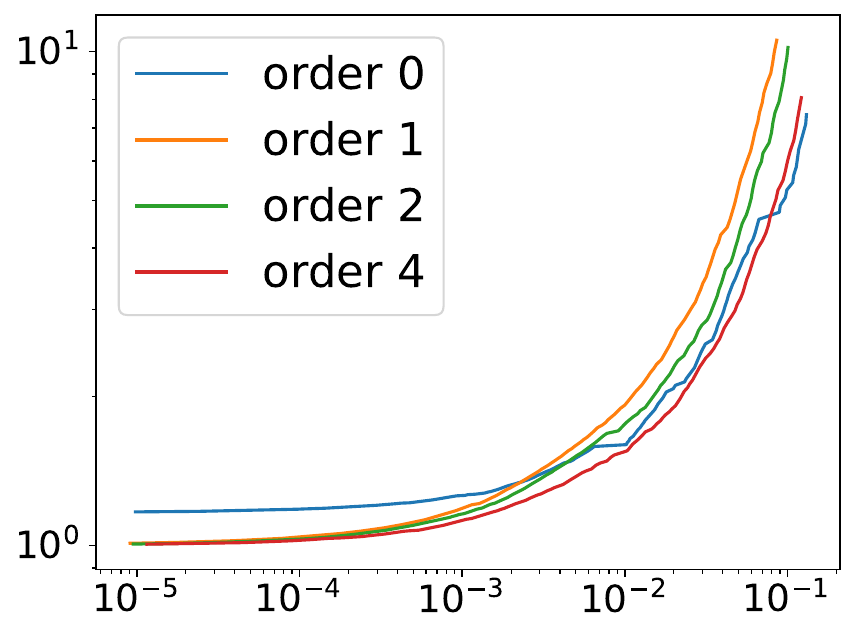}}
    \\[0.5em]
    \stackinset{c}{0pt}{t}{-12pt}{}{\includegraphics[height=\w]{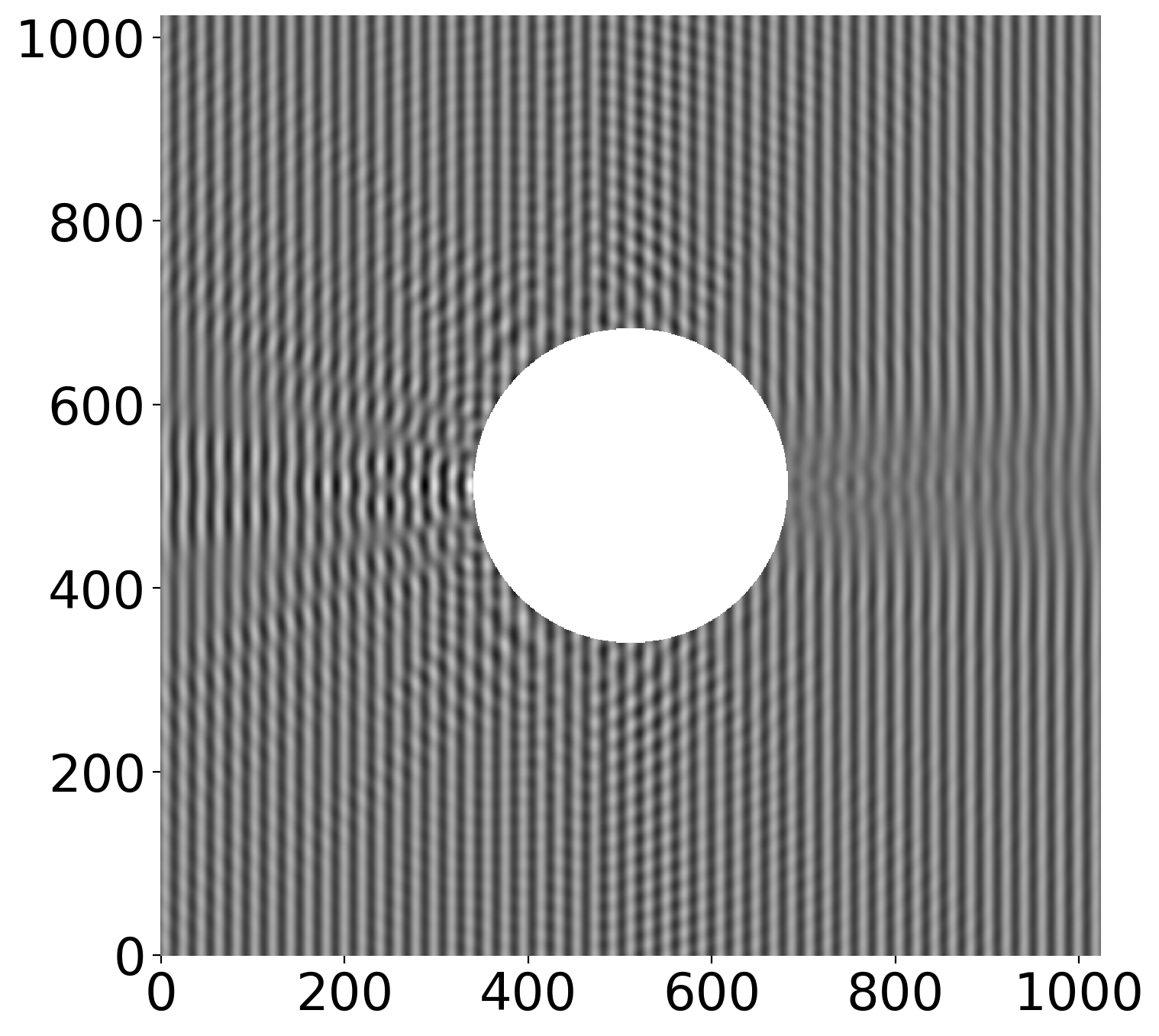}}
    \quad
    \stackinset{c}{0pt}{t}{-12pt}{}{\includegraphics[height=\w]{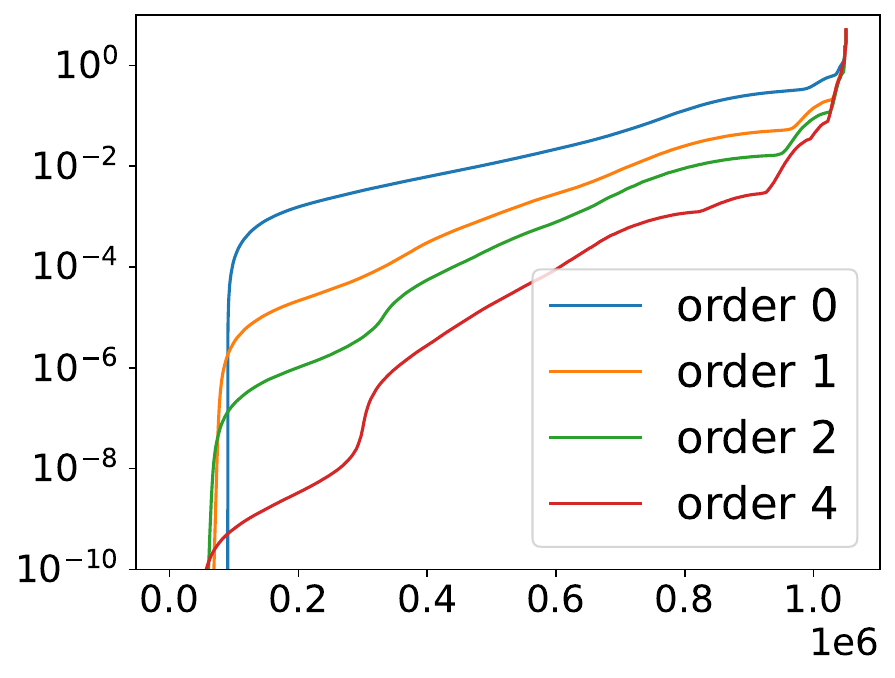}}
    \quad
    \stackinset{c}{0pt}{t}{-12pt}{}{\includegraphics[height=\w]{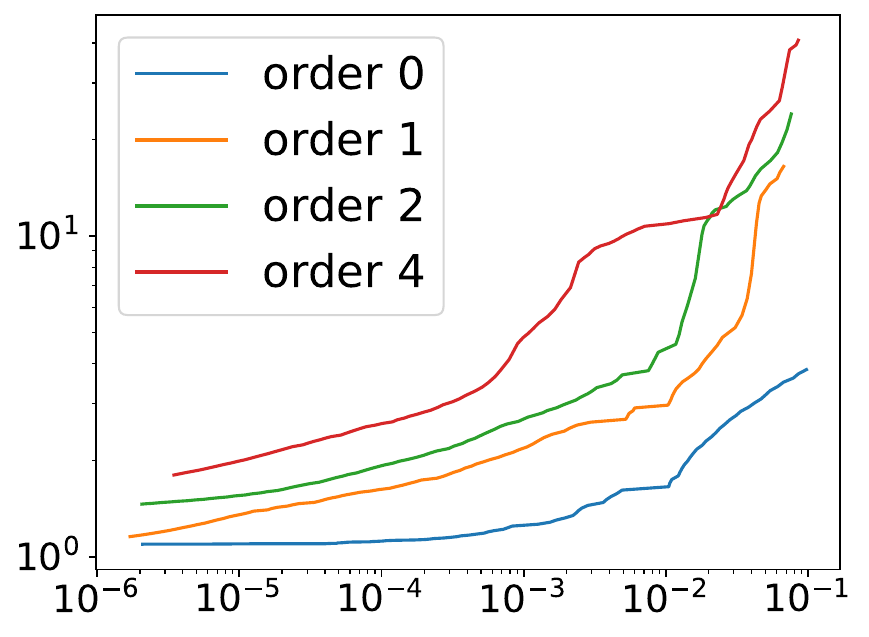}}
    \caption{Impact of the order on the compression ratio.}
    \label{fig:ex2_image}
\end{figure}

\section{Conclusion and future work}
\label{sec:conclusion}

In this work, we reformulated the MGARD algorithm as a wavelet transform on the interval. 
This reformulation provides new insights into the design and analysis of algorithm extensions. 
Specifically, we proposed a new family of projection operators that lead to compactly supported wavelet basis functions without compromising the stability of the transform. 
Moreover, the proposed formulation enables the construction of wavelets with an arbitrary number of vanishing moments by design. 
We demonstrated that this property results in better compression of data with matching regularity.
Several questions remain to be addressed in future works, including data-adaptive order selection, optimal quantization, and rigorous error analysis.

\section*{Acknowledgments}
This research was supported by the SIRIUS-2 ASCR research project, the Scientific Discovery through Advanced Computing (SciDAC) program, specifically the RAPIDS-2 and FastMath SciDAC institutes, and the GE-ORNL CRADA data reduction project. 
This research used resources of the Oak Ridge Leadership Computing Facility, which is a DOE Office of Science User Facility.









\medskip
Received xxxx 20xx; revised xxxx 20xx; early access xxxx 20xx.
\medskip

\appendix
\def\grid{\mathcal{G}}

\section{Appendix}

\subsection{Polynomial interpolation on dyadically nested uniform grids}
\label{app:interpmat}

\begin{lemma}\label{lemma:interpmat}
    Given two nested one-dimensional uniform grids $\grid_{0}\subset\grid_{1}$ with $(r, 2r)$ elements, $(rq, 2rq)$ intervals, and $(rq+1, 2rq+1)$ nodes respectively, the $q$-th order polynomial interpolation matrix $\predictor^{r,q}$ mapping the values $\{u_{2n}:n=0,\hdots,rq\}$ at the nodes of $\grid_{0}$ to the values $\{u_{2m+1}:m=0,\hdots,rq-1\}$ at the nodes of the surplus grid $\grid_{1}\setminus\grid_{0}$ has the form
    \begin{align*}
        &\predictor^{r,q} = \underbrace{\predictor^{q} \overt \hdots \overt \predictor^{q}}_{r\text{ times}},
        \qquad r>0, \; q\geq 0,
    \end{align*}
    with $\overt$ defined in Definition~\ref{def:stackelemets}, $\predictor^0 = [1]$, and
    \begin{align*}
        [\predictor^q]_{m,n}
        = \frac{(-1)^{m+n}}{2^q(1+2m-2n)} \frac{(2m+1)!!(2q-2m-1)!!}{n!(q-n)!},
        \qquad
        m&=0,\hdots,q-1,
        \\
        n&=0,\hdots,q,
        \quad q > 0,
    \end{align*}
    where $p!!=1\cdot 3\cdot\hdots\cdot(p-2)p$ is the double factorial.
\end{lemma}
\begin{proof}
    The case of $q=0$ is trivial. 
    For $q>0$, the values of the Lagrange interpolant $\mathcal{I}^qu$ at the nodes of $\grid_{1}\setminus\grid_{0}$ are given by
    \begin{align*}
        [\mathcal{I}^qu]_{2m+1}
        &:= \sum_{n=0}^{q} u_{2n} \prod_{\substack{k=0\\k\neq n}}^{q} \frac{x_{2m+1}-x_{2k}}{x_{2n}-x_{2k}}
        \\
        &= \big[ \predictor^q \cdot [u_0,u_2,\hdots,u_{2q}]^T \big]_{2m+1},
        \quad
        m=0,\hdots,q-1.
    \end{align*}
    For a uniform grid with $x_n=nh$, we get
    \begin{align*}
        [\predictor^q]_{m,n}
        &:= \prod_{\substack{k=0\\k\neq n}}^{q} \frac{x_{2m+1}-x_{2k}}{x_{2n}-x_{2k}}
        = \prod_{\substack{k=0\\k\neq n}}^{q} \frac{2(m-k)+1}{2(n-k)}.
    \end{align*}
    To compute the denominator, collect the positive and negative terms in the product to get
    \begin{align*}
        &\prod_{\substack{k=0\\k\neq n}}^{q} \frac{1}{2(n-k)}
        = \frac{1}{2^q} \prod_{k=0}^{n-1} \frac{1}{n-k} \prod_{k=n+1}^{q} \frac{1}{n-k}
        = \frac{1}{2^q} \frac{(-1)^{q-n}}{n!(q-n)!}.
    \end{align*}
    Similarly, by collecting the positive and negative terms in the numerator, one gets
    \begin{align*}
        \prod_{\substack{k=0\\k\neq n}}^{q} (2(m-k)+1)
        &= \frac{(-1)^{q-m}(2m+1)!!(2q-2m-1)!!}{1-2n+2m}.
    \end{align*}
\end{proof}

\subsection{Gramm matrices for the elements of dyadically nested uniform grids.}
\label{app:grammmat}

\begin{lemma}\label{lemma:updatemat}
    Under the same conditions as in Lemma~\ref{lemma:interpmat}, the $q$-th order update matrix $\matop{U}^{r,q}$ mapping the values $\{u_{2m+1}:m=0,\hdots,rq-1\}$ at the nodes of the surplus grid $\grid_{1}\setminus\grid_{0}$ to the values $\{u_{2n}:n=0,\hdots,rq\}$ at the nodes of $\grid_{0}$ has the form
    \begin{align*}
        \matop{U}^{r,q} := \left(\matop{G}^{r,q}_{\scriptscriptstyle\Delta\Delta}\right)^{-1} \matop{G}^{r,q}_{\scriptscriptstyle\Delta\nabla}
    \end{align*}
    where
    \begin{align*}
        \matop{G}^{r,q}_{\scriptscriptstyle\Delta\Delta} &= 
        \underbrace{\matop{G}^{q}_{\scriptscriptstyle\Delta\Delta} \oplus \hdots \oplus \matop{G}^{q}_{\scriptscriptstyle\Delta\Delta}}_{r\text{ times}},
        \qquad
        \matop{G}^{r,q}_{\scriptscriptstyle\Delta\nabla} = 
        \underbrace{\matop{G}^{q}_{\scriptscriptstyle\Delta\nabla} \ominus \hdots \ominus \matop{G}^{q}_{\scriptscriptstyle\Delta\nabla}}_{r\text{ times}},
        \qquad r>0, \; q\geq 0
    \end{align*}
    with $\oplus$ and $\ominus$ defined in Definition~\ref{def:stackelemets}.
    The entries of the Gramm matrices are given by
    \begin{align*}
        \matop{G}_{\scriptscriptstyle\Delta\Delta}^0 = [2h],
        \qquad
        \matop{G}_{\scriptscriptstyle\Delta\nabla}^0 = [h]
    \end{align*}
    and, for $q>0$,
    {\small
    \begin{align*}
        &[\matop{G}_{\scriptscriptstyle\Delta\Delta}^q]_{n,m}
        = \frac{2h(-1)^{n+m}}{n!m!(q-n)!(q-m)!} 
        \int\displaylimits_0^{q} \frac{\prod_{\substack{k=0}}^{q} (x-k)^2}{(x-n)(x-m)} dx,
        \quad n=0,\hdots,q,
        \; m=0,\hdots,q,
        \\
        &[\matop{G}_{\scriptscriptstyle\Delta\nabla}^q]_{n,m}
        = 
        \frac{h(-1)^{n+1}}{2^qn!(q-n)!}
        \\
        &\times
        \begin{cases}
            \displaystyle
            \frac{1}{(2m+1)!(q-2m-1)!} 
            \int\displaylimits_{0}^{q} \frac{\prod_{\substack{k=0}}^{q} (x-2k) (x-k)}{(x-2n)(x-2m-1)} dx
            & m\leq q/2-1,
            \\[2em]
            \displaystyle
            \frac{1}{q!}
            \left[ 
            \int\displaylimits_{0}^{q} \frac{\prod_{\substack{k=0}}^{q} (x-2k) (x-k)}{(x-2n)(x-2m-1)} dx - 
            \int\displaylimits_{q}^{2q} \frac{\prod_{\substack{k=0}}^{q} (x-2k) (x-k-q)}{(x-2n)(x-2m-1)} dx
            \right]
            & \displaystyle m=\frac{q-1}{2},
            \\[2em]
            \displaystyle
            \frac{(-1)^{q}}{(2m+1-q)!(2q-2m-1)!} 
            \int\displaylimits_{q}^{2q} \frac{\prod_{\substack{k=0}}^{q} (x-2k) (x-k-q)}{(x-2n)(x-2m-1)} dx
            & m\geq q/2
        \end{cases}
    \end{align*}
    }
    for $n=0,\hdots,q$, $m=0,\hdots,q-1$.
\end{lemma}
\begin{proof}
    The case of $q=0$ is trivial.
    For $q>0$, we get from the definition of Gram matrices
    \begin{alignat*}{2}
        [\matop{G}_{\scriptscriptstyle\Delta\Delta}]_{n,m} &:= \langle\varphi_{0,2n},\varphi_{0,2m}\rangle,
        &&\quad n = 0,\hdots,q, \quad m = 0,\hdots,q,
        \\
        [\matop{G}_{\scriptscriptstyle\Delta\nabla}]_{n,m} &:= \langle\varphi_{0,2n},\varphi_{1,2m+1}\rangle,
        &&\quad n = 0,\hdots,q, \quad m = 0,\hdots,q-1,
    \end{alignat*}
    where $\varphi_{0}$, $\varphi_{1}$ are the Lagrange nodal basis functions on the grids $\grid_0$, $\grid_1$ respectively.
    For a uniform grid with $x_n=nh$, we get
    \begin{align*}
        \varphi_{0,2n}(x) &= \prod_{\substack{k=0\\k\neq n}}^{q} \frac{x-2kh}{2(n-k)h},
        \qquad n=0,\hdots,q,
        \quad x\in[0,2hq],
    \end{align*}
    and
    {\small
    \begin{align*}
        \varphi_{1,2m+1} &= 
        \begin{cases}
            \displaystyle{\prod_{\substack{k=0\\k\neq 2m+1}}^{q} \frac{x-kh}{(2m+1-k)h}},   
            & m=0,\hdots,\frac{q-1-(q-1)\%2}{2}, \quad x\in[0,qh],
            \\[1.5em]
            \displaystyle{\prod_{\substack{k=0\\k+q\neq 2m+1}}^{q} \frac{x-(k+q)h}{(2m+1-q-k)h}}, 
            & m=\frac{q-1+(q-1)\%2}{2},\hdots,q-1, \quad x\in[qh,2qh],
        \end{cases}
    \end{align*}
    }
    where $(q-1)\%2\in\{0,1\}$ is the division remainder, i.e., $(q-1)\%2=0$ for odd $q$ and $(q-1)\%2=1$ otherwise.
    Hence, we have 
    \begin{align*}
        [\matop{G}_{\scriptscriptstyle\Delta\Delta}]_{n,m}
        &= \int_0^{2hq} \prod_{\substack{k=0\\k\neq n}}^{q} \frac{x-2kh}{2(n-k)h} \prod_{\substack{k=0\\k\neq m}}^{q} \frac{x-2kh}{2(m-k)h} dx
        = 2h \int_0^{q} \prod_{\substack{k=0\\k\neq n}}^{q} \frac{x-k}{n-k} \prod_{\substack{k=0\\k\neq m}}^{q} \frac{x-k}{m-k} dx
        \\
        &= 2h \frac{(-1)^{n+m}}{n!m!(q-n)!(q-m)!} \int_0^{q} 
        \frac{\prod_{\substack{k=0}}^{q} (x-k)^2}{(x-n)(x-m)} dx.
    \end{align*}
    Similarly, for $m\leq q/2-1$, we have
    \begin{align*}
        [\matop{G}_{\scriptscriptstyle\Delta\nabla}]_{n,m} &:= 
        \int_0^{qh} 
        \prod_{\substack{k=0\\k\neq n}}^{q} \frac{x-2kh}{2(n-k)h} 
        \prod_{\substack{k=0\\k\neq 2m+1}}^{q} \frac{x-kh}{(2m+1-k)h} 
        dx
        \\
        &=
        \frac{h}{2^q} \int_0^{q} 
        \prod_{\substack{k=0\\k\neq n}}^{q} \frac{x-2k}{n-k} 
        \prod_{\substack{k=0\\k\neq 2m+1}}^{q} \frac{x-k}{2m+1-k} 
        dx
        \\
        &= \frac{h}{2^q} \frac{(-1)^{n+1}}{n!(2m+1)!(q-n)!(q-2m-1)!} 
        \int_0^{q} 
        \frac{\prod_{\substack{k=0}}^{q} (x-2k) (x-k)}{(x-2n)(x-2m-1)}
        dx
    \end{align*}
    and for $m\geq q/2$, we get
    \begin{align*}
        &[\matop{G}_{\scriptscriptstyle\Delta\nabla}]_{n,m} := 
        \int_{qh}^{2qh}
        \prod_{\substack{k=0\\k\neq n}}^{q} \frac{x-2kh}{2(n-k)h} 
        \prod_{\substack{k=0\\k+q\neq 2m+1}}^{q} \frac{x-(k+q)h}{(2m+1-q-k)h} 
        dx
        \\
        &= 
        \frac{h}{2^q} \int_{q}^{2q}
        \prod_{\substack{k=0\\k\neq n}}^{q} \frac{x-2k}{n-k} 
        \prod_{\substack{k=0\\k+q\neq 2m+1}}^{q} \frac{x-(k+q)}{2m+1-q-k} 
        dx
        \\
        &=
        \frac{h}{2^q} \frac{(-1)^{n+q+1}}{n!(2m+1-q)!(q-n)!(2q-2m-1)!} 
        \int\limits_{q}^{2q} 
        \frac{\prod_{\substack{k=0}}^{q} (x-2k) (x-k-q)}{(x-2n)(x-2m-1)}
        dx.
    \end{align*}
    For $q$ odd and $m=\frac{q-1}{2}$, we also have
    \begin{align*}
        &[\matop{G}_{\scriptscriptstyle\Delta\nabla}]_{n,m} := 
        \frac{h}{2^q} \int_0^{q} 
        \prod_{\substack{k=0\\k\neq n}}^{q} \frac{x-2k}{n-k} 
        \prod_{\substack{k=0\\k\neq 2m+1}}^{q} \frac{x-k}{2m+1-k} 
        dx
        \\
        &+
        \frac{h}{2^q} \int_{q}^{2q}
        \prod_{\substack{k=0\\k\neq n}}^{q} \frac{x-2k}{n-k} 
        \prod_{\substack{k=0\\k+q\neq 2m+1}}^{q} \frac{x-(k+q)}{2m+1-q-k} 
        dx
        \\
        &=
        \frac{h}{2^q} \frac{(-1)^{n+1}}{n!(q-n)!(2m+1)!(q-2m-1)!}
        \left[ 
        \int_0^{q} 
        \frac{\prod_{\substack{k=0}}^{q} (x-2k) (x-k)}{(x-2n)(x-2m-1)}
        dx
        \right.
        \\[1em]
        &+
        (-1)^q
        \left.
        \frac{(2m+1)!(q-2m-1)!}{(2m+1-q)!(2q-2m-1)!} 
        \int_{q}^{2q} 
        \frac{\prod_{\substack{k=0}}^{q} (x-2k) (x-k-q)}{(x-2n)(x-2m-1)}
        dx
        \right]
        \\[1em]
        &= \frac{h}{2^q} \frac{(-1)^{n+1}}{n!q!(q-n)!}
        \left[ 
        \int\limits_0^{q} 
        \frac{\prod_{\substack{k=0}}^{q} (x-2k) (x-k)}{(x-2n)(x-q)}
        dx
        - \int\limits_{q}^{2q} 
        \frac{\prod_{\substack{k=0}}^{q} (x-2k) (x-k-q)}{(x-2n)(x-q)}
        dx
        \right].
    \end{align*}
\end{proof}


\begin{thebibliography}{99}

\bibitem{abramowitz1988handbook}
\newblock M. Abramowitz, I. Stegun and R. Romer,
\newblock \emph{Handbook of mathematical functions with formulas, graphs, and mathematical tables},
\newblock American Association of Physics Teachers, College Park, MD, 1988.

\bibitem{AinsworthMGARD1D2018}
\newblock M. Ainsworth, O. Tugluk, B. Whitney and S. Klasky,
\newblock \doititle{Multilevel techniques for compression and reduction of scientific data -- the univariate case},
\newblock \emph{Comput. Visualization Sci.}, \textbf{19(5-6)} (2018), 65-76.

\bibitem{AinsworthMGARDnD2019}
\newblock M. Ainsworth, O. Tugluk, B. Whitney and S. Klasky,
\newblock \doititle{Multilevel techniques for compression and reduction of scientific data -- the multivariate case},
\newblock \emph{SIAM J. Sci. Comput.}, \textbf{41(2)} (2019), A1278-A1303.

\bibitem{AinsworthMGARDQoI2019}
\newblock M. Ainsworth, O. Tugluk, B. Whitney and S. Klasky,
\newblock \doititle{Multilevel techniques for compression and reduction of scientific data-quantitative control of accuracy in derived quantities},
\newblock \emph{SIAM J. Sci. Comput.}, \textbf{41(4)} (2019), A2146-A2171.

\bibitem{AinsworthMGARDFE2020}
\newblock M. Ainsworth, O. Tugluk, B. Whitney and S. Klasky,
\newblock \doititle{Multilevel techniques for compression and reduction of scientific data -- the unstructured case},
\newblock \emph{SIAM J. Sci. Comput.}, \textbf{42(2)} (2020), A1402-A1427.

\bibitem{betcke2021bempp}
\newblock T. Betcke and M. Scroggs,
\newblock \doititle{Bempp-cl: A fast Python based just-in-time compiling boundary element library},
\newblock \emph{J. Open Source Softw.}, \textbf{59(6)} (2021), 2879--2879.

\bibitem{BornemannNorm1993}
\newblock F. Bornemann and H. Yserentant,
\newblock \doititle{A basic norm equivalence for the theory of multilevel methods},
\newblock \emph{Numer. Math.}, \textbf{64(1)} (1993), 455-476.

\bibitem{BURT1987671}
\newblock P. Burt and E. Adelson,
\newblock The {L}aplacian pyramid as a compact image code,
\newblock \emph{Readings in Computer Vision}, Morgan Kaufmann, San Francisco, CA, 1987, 671-679.

\bibitem{cohen2003numerical}
\newblock A. Cohen,
\newblock \emph{Numerical analysis of wavelet methods},
\newblock Elsevier, New York, 2003.

\bibitem{Dahmen1995}
\newblock W. Dahmen,
\newblock \doititle{Stability of multiscale transformations},
\newblock \emph{J. Fourier Anal. Appl. [Elektronische Ressource]}, \textbf{2} (1995), 341-362.

\bibitem{daubechies1998factoring}
\newblock I. Daubechies and W. Sweldens,
\newblock \doititle{Factoring wavelet transforms into lifting steps},
\newblock \emph{J.Fourier Anal. Appl.}, \textbf{4} (1998), 247-269.

\bibitem{DeVore1992_smoothness}
\newblock R. DeVore, B. Jawerth and B. Lucier,
\newblock \doititle{Image compression through wavelet transform coding},
\newblock \emph{IEEE Trans. Inf. Theory}, \textbf{38(2)} (1992), 719-746.

\bibitem{DeVore1992_smoothness_1994}
\newblock R. DeVore and B. Lucier,
\newblock \doititle{Classifying the smoothness of images: theory and applications to wavelet image processing},
\newblock \emph{Proceedings of 1st International Conference on Image Processing}, Austin, TX, 1994, 6-10.

\bibitem{fliege1994multirate}
\newblock N. Fliege,
\newblock \emph{Multirate digital signal processing: multirate systems, filter banks, wavelets},
\newblock John Wiley \& Sons, Inc., Hoboken, NJ, 1994.

\bibitem{FLOATER2005623}
\newblock M. Floater, G. K\'{o}s and M. Reimers,
\newblock \doititle{Mean value coordinates in 3D},
\newblock \emph{Comput. Aided Geom. Des.}, \textbf{22(7)} (2005), 623-631.

\bibitem{GongMGARDSoftware2023}
\newblock Q. Gong, J. Chen, B. Whitney, X. Liang, V. Reshniak, T. Banerjee, J. Lee, A. Rangarajan, L. Wan, N. Vidal, Q. Liu, A. Gainaru, N. Podhorszki, R. Archibald, S. Ranka and S. Klasky,
\newblock \doititle{{MGARD}: A multigrid framework for high-performance, error-controlled data compression and refactoring},
\newblock \emph{SoftwareX}, \textbf{24} (2023).

\bibitem{GongClimate2023}
\newblock Q. Gong, C. Zhang, X. Liang, V. Reshniak, J. Chen, A. Rangarajan, S. Ranka, N. Vidal, L. Wan, P. Ullrich, N. Podhorszki, R. Jacob and S. Klasky,
\newblock \doititle{Spatiotemporally adaptive compression for scientific dataset with feature preservation -- a case study on simulation data with extreme climate events analysis},
\newblock \emph{2023 IEEE 19th International Conference on e-Science (e-Science)}, Limassol, Cyprus, 2023, 1-10.

\bibitem{HACKEMACK2018188}
\newblock M. Hackemack and J. Ragusa,
\newblock \doititle{Quadratic serendipity discontinuous finite element discretization for SN transport on arbitrary polygonal grids},
\newblock \emph{J. Comput. Phys.	}, \textbf{374} (2018).

\bibitem{Kolomenskiy2022-po}
\newblock D. Kolomenskiy, R. Onishi and H. Uehara,
\newblock \doititle{{WaveRange}: wavelet-based data compression for three-dimensional numerical simulations on regular grids},
\newblock \emph{J. Visualization	}, \textbf{25(3)} (2022), 543-573.

\bibitem{lakshminarasimhan2013isabela}
\newblock S. Lakshminarasimhan, N. Shah, S. Ethier, S. Ku, C. Chang, S. Klasky, R. Latham, R. Ross and N. Samatova,
\newblock \doititle{{ISABELA} for effective in situ compression of scientific data},
\newblock \emph{Concurrency Comput. Pract. Exper.	}, \textbf{25(4)} (2013), 524-540.

\bibitem{Lee2022}
\newblock J. Lee, Q. Gong, J. Choi, T. Banerjee, S. Klasky, S. Ranka and A. Rangarajan,
\newblock \doititle{Error-bounded learned scientific data compression with preservation of derived quantities},
\newblock \emph{Appl. Sci.}, \textbf{12(13)} (2022).

\bibitem{Mallat1989}
\newblock S. Mallat,
\newblock \doititle{A theory for multiresolution signal decomposition: the wavelet representation},
\newblock \emph{IEEE Trans. Pattern Anal. Mach. Intell.}, \textbf{11(7)} (1989), 674-693.

\bibitem{Minh2012}
\newblock N. Minh and M. Yue,
\newblock \doititle{Multidimensional filter banks and multiscale geometric representations},
\newblock \emph{Found. Trends Signal Process.}, \textbf{5(3)} (2012), 157-264.

\bibitem{Oswald1994}
\newblock P. Oswald,
\newblock \emph{Multilevel finite element approximation},
\newblock Vieweg+Teubner Verlag Wiesbaden, 1994.

\bibitem{Pearlman2004}
\newblock W. Pearlman, A. Islam, N. Nagaraj and A. Said,
\newblock \doititle{Efficient, low-complexity image coding with a set-partitioning embedded block coder},
\newblock \emph{IEEE Trans. Circuits Syst. Video Technol.}, \textbf{44(11)} (2004), 1219-1235.

\bibitem{Peyrot2019}
\newblock J. Peyrot, L. Duval, F. Payan, L. Bouard, L. Chizat, S. Schneider and M. Antonini,
\newblock \doititle{{HexaShrink}, an exact scalable framework for hexahedral meshes with attributes and discontinuities: multiresolution rendering and storage of geoscience models},
\newblock \emph{Comput. Geosci.}, \textbf{23(4)} (2019), 723-743.

\bibitem{Roos1988}
\newblock P. Roos, M. Viergever, M. van Dijke and J. Peters,
\newblock \doititle{Reversible intraframe compression of medical images},
\newblock \emph{IEEE Trans. Med. Imaging}, \textbf{7(4)} (1988), 328-336.

\bibitem{SAYOOD20061}
\newblock K. Sayood,
\newblock \emph{Introduction to Data Compression},
\newblock The Morgan Kaufmann Series in Multimedia Information and Systems, Morgan Kaufmann, San Francisco, CA, 2006.

\bibitem{Sermutlu_2022}
\newblock E. Sermutlu,
\newblock \doititle{A close look at {Newton–Cotes} integration rules},
\newblock \emph{Results Nonlinear Anal.}, \textbf{2(2)} (2022), 48-60.

\bibitem{Sperr2023}
\newblock S. Li, P. Lindstrom and J. Clyne,
\newblock \doititle{Lossy scientific data compression with {SPERR}},
\newblock \emph{2023 IEEE International Parallel and Distributed Processing Symposium (IPDPS)}, Lyon, France, 2023, 1007-1017.

\bibitem{STEPHANE20091}
\newblock S. Mallat,
\newblock \emph{A Wavelet tour of signal processing},
\newblock Academic Press, San Diego, CA, 2009.

\bibitem{Sweldens2000}
\newblock W. Sweldens and P. Schr\"{o}der,
\newblock Building your own wavelets at home,
\newblock \emph{Wavelets in the Geosciences}, Springer, Berlin, Heidelberg, 2005, 72-107.

\bibitem{Taubman2002}
\newblock D. Taubman and M. Marcellin,
\newblock \emph{{JPEG}2000 image compression fundamentals, standards and practice},
\newblock The Springer International Series in Engineering and Computer Science, Springer New York, NY, 2002.

\bibitem{Vaidyanathan1990}
\newblock P. Vaidyanathan,
\newblock \doititle{Multirate digital filters, filter banks, polyphase networks, and applications: a tutorial},
\newblock \emph{Proc. IEEE}, \textbf{78(1)} (1990), 56-93.

\bibitem{Viergever1993}
\newblock M. Viergever and P. Roos,
\newblock \doititle{Hierarchical interpolation},
\newblock \emph{IEEE Eng. Med. Biol. Mag.	}, \textbf{12(1)} (1993), 48-55.

\bibitem{XuSubspace1992}
\newblock J. Xu,
\newblock \doititle{Iterative methods by space decomposition and subspace correction},
\newblock \emph{SIAM Rev.}, \textbf{34(4)} (1992), 581-613.

\bibitem{Zhao2020}
\newblock K. Zhao, S. Di, X. Lian, S. Li, D. Tao, J. Bessac, Z. Chen and Franck Cappello,
\newblock \doititle{{SDRBench}: Scientific data reduction benchmark for lossy compressors},
\newblock \emph{2020 IEEE International Conference on Big Data (Big Data)}, 2020, 2716-2724.

\end{thebibliography}
\end{document}